%% file: articulo_arxiv.tex
\documentclass[11pt]{amsart}
\usepackage{amsfonts,amssymb,amsmath,amsthm}
\usepackage{anysize}
\usepackage{multirow,array,mathtools,tabularx}
\usepackage[T1]{fontenc}
\usepackage{mathrsfs,bbm,textcomp}
\usepackage[all]{xy}
\usepackage{float}
\usepackage{graphicx,wrapfig,rotating}
\usepackage{caption}
\usepackage{subcaption}
\usepackage[shortlabels]{enumitem}
\usepackage{aliascnt} 
\usepackage{fancyhdr,url,hyperref}
\hypersetup{pdfauthor={Juanita Pinz\'on Caicedo},colorlinks=true,linkcolor=Cyan,citecolor 	=Cyan}
\usepackage[numbers,sort]{natbib}
\usepackage{setspace}
\usepackage[usenames,dvipsnames]{xcolor}
\usepackage{stackrel}

\marginsize{0.9in}{0.9in}{0.9in}{0.9in}

\floatstyle{ruled}
\restylefloat{figure}

\newlength{\mytextsize}
\makeatletter
\def\MyNewTheorem#1[#2]#3{%
  \newaliascnt{#1}{#2}
  \newtheorem{#1}[#1]{#3}
  \aliascntresetthe{#1}
  \expandafter\newcommand\csname #1autorefname\endcsname{#3}
}
\makeatother

\newtheorem{theorem}{Theorem}
\MyNewTheorem{proposition}[theorem]{Proposition}
\MyNewTheorem{cor}[theorem]{Corollary}
\MyNewTheorem{lemma}[theorem]{Lemma}
\MyNewTheorem{claim}[theorem]{Claim}
\MyNewTheorem{rem}[theorem]{Remark}
\MyNewTheorem{definition}[theorem]{Definition}
\newtheorem*{thm*}{Theorem}

\makeatletter
\newtheorem*{rep@theorem}{\rep@title}
\newcommand{\newreptheorem}[2]{%
\newenvironment{rep#1}[1]{%
 \def\rep@title{#2 \ref{##1}}%
 \begin{rep@theorem}}%
 {\end{rep@theorem}}}
\makeatother

\newreptheorem{theorem}{Theorem}
\newreptheorem{lemma}{Lemma}

\def\equationautorefname~#1\null{(#1)\null}
\def\itemautorefname~#1\null{#1\null}

\newcommand{\Z}{{\ensuremath{\mathbb{Z}}}}

\newcommand{\M}{{\ensuremath{\mathscr{M}}}}

\newcommand{\Conc}{{\ensuremath{\mathcal{C}}}}

\newcommand{\rst}[1]{\ensuremath{{\big|} \raise-1.25ex\hbox{$\scriptscriptstyle#1$}}}

\newcommand{\bigsharpp}{\mathop{\mathchoice
  {\vcenter{\hbox{\LARGE$\#$}}}
  {\vcenter{\hbox{\large$\#$}}}
  {\vcenter{\hbox{\footnotesize$\#$}}}
  {\vcenter{\hbox{\scriptsize$\#$}}}
}\displaylimits}

\newcommand{\union}[1]{\underset{\raisebox{-0.125ex}[0pt][0pt]{\text{\scriptsize{$#1$} }}}{\raisebox{0.25ex}[0pt][0pt]{\text{ $\cup$ }}}}

\title{Independence of Satellites of Torus Knots in the Smooth Concordance Group}
\author{Juanita Pinz\'on-Caicedo}
\address{Department of Mathematics, University of Georgia,
Athens, GA \ 30605}
\email{jpinzon@uga.edu}
\date{}

\begin{document}
\begin{abstract} The main goal of this article is to obtain a condition under which an infinite collection $\mathscr{F}$ of satellite knots (with companion a positive torus knot and pattern similar to the Whitehead link) freely generates a subgroup of infinite rank in the smooth concordance group. This goal is attained by examining both the instanton moduli space over a 4-manifold with tubular ends and the corresponding Chern-Simons invariant of the adequate 3-dimensional portion  of the 4-manifold. More specifically, the result is derived from Furuta's criterion for the independence of Seifert fibred homology spheres in the homology cobordism group of oriented homology 3-spheres. Indeed, we first associate to $\mathscr{F}$ the corresponding collection of 2-fold covers of the 3-sphere branched over the elements of $\mathscr{F}$ and then introduce definite cobordisms from the aforementioned covers of the satellites to a number of Seifert fibered homology spheres. This allows us to apply Furuta's criterion and thus obtain a condition that guarantees the independence of the family $\mathscr{F}$ in the smooth concordance group.
\end{abstract}

\maketitle
\section{Introduction} 
A knot is a smooth embedding of $S^1$ into $S^3$.
Two knots $K_0$ and $K_1$ are said to be smoothly concordant if there is a
smooth embedding of $S^1\times [0,1]$ into $S^3\times [0,1]$ that
restricts to the given knots at each end. Requiring such an embedding to be locally flat instead of smooth gives rise to the weaker notion of topological concordance.  Both kinds of concordance are equivalence relations, and the sets of smooth and topological concordance classes of knots,
denoted by $\mathcal{C}_\infty$ and $\mathcal{C}_{\mbox{\tiny TOP}}$ respectively, are abelian groups with connected sum as their binary operation. In both cases the identity element is the concordance class of the unknot and the knots in that class are known respectively as smoothly slice and topologically slice. The algebraic structure of $\mathcal{C}_\infty$ and $\mathcal{C}_{\mbox{\tiny TOP}}$ is a much studied object in low-dimensional topology, as is the concordance class of the unknot. Identifying the set of knots that are topologically slice but not smoothly slice is an intriguing topic, among other reasons because these knots reveal subtle properties of differentiable structures in dimension four \cite[p. 522]{gompf-stip}.\\

One way to approach the subject is by studying the effect of special kinds of satellite operations on concordance. To define a satellite operation we start with a given knot $B$, embedded in an unkotted solid torus $V\subseteq S^3$, and a second knot $K\subseteq S^3$. The satellite knot with pattern $B\subseteq V$ and companion $K$ is denoted by $B(K)$ and is obtained as the image of $B$ under the embedding of $V$ in $S^3$ that knots $V$ as a tubular neighborhood of $K$. Freedman's theorem  \cite{freedman, freedman2, freedman-quinn} implies that if the pattern $B$ is an unknot in $S^3$ and is trivial in $H_1(V;\Z)$, then the satellite $B(K)$ is topologically slice, that is, it maps to zero in $\mathcal{C}_{\mbox{\tiny TOP}}$.\\

Whitehead doubles are an important example of such satellites and are obtained by using the Whitehead link (\autoref{Whitehead}) as the pattern of the operation. Similar examples arise by considering Whitehead-like patterns $D_n$ (\autoref{fig::D_n}). Because the knot $D_n$ is trivial in $S^3$, every satellite knot with pattern $D_n$ is topologically slice, and no classical invariant captures information about their smooth concordance type. Thus, smooth techniques like gauge theory are necessary to obtain that information. In this article we use the theory of $SO(3)$ instantons to establish an obstruction for a family of Whitehead-like satellites of positive torus knots to be independent in the smooth concordance group. The main result is the following:

\begin{reptheorem}{main_result}
Let $\left\{\left(p_i,q_i\right)\right\}_i$ be a sequence of relatively prime positive integers and $n_i$ a positive and even integer $(i=1,2,\ldots)$. Then, if $$p_iq_i(2n_ip_iq_i-1)<p_{i+1}q_{i+1}(n_{i+1}p_{i+1}q_{i+1}-1),$$ the collection $\left\{D_{n_i}\left(T_{p_iq_i}\right)\right\}_{i=1}^\infty$ is an independent family in $\Conc_\infty$.
\end{reptheorem}

It is important to mention that the case $n_i=2$ is a result of Hedden-Kirk \cite{hedden-kirk} and the previous theorem is a generalization of their work.\\

The proof of \autoref{main_result} is based on a technique pioneered by Akbulut \cite{akbulut} and later expanded by Cochran-Gompf \cite{cochran-gompf} among others. The starting point of Akbulut's technique is to assign to each satellite knot $D_n(T_{p,q})$ the 2-fold cover of $S^3$ branched over the knot $D_n(T_{p,q})$, since an obstruction to the cover from bounding results in an obstruction to $D_n(T_{p,q})$ from being slice. The next step is to construct a negative definite cobordism $W$ from the 2-fold cover $\Sigma=\Sigma_2\left(D_n(T_{p,q})\right)$ to the Seifert fibered homology sphere $\Sigma(2,3,5)$ and then glue $W$ to the negative definite 4-manifold $E_8$ along their common boundary. The last step is to notice that if $\Sigma$ bounded a $\Z/2$--homology 4-ball $Q$, then the manifold $X=Q\cup W\cup E_8$ would be a closed 4-manifold with negative definite intersection form given by $m\langle -1\rangle\oplus E_8$ for some integer $m>0$. However, Donaldson's Diagonalization Theorem prevents the existence of such manifold thus showing that $Q$ cannot exist and that $D_n(T_{p,q})$ is not smoothly slice. In fact, the technique is more powerful than that: it can be used to prove that $D_n(T_{p,q})$ has infinite order in the smooth concordance group $\Conc_\infty$. Nevertheless, Donaldson's theorem is not powerful enough to prove independence of infinite families. In that case, it is necessary to use a generalization of the following theorem:

\begin{thm*}[Furuta \cite{furuta}, see also \cite{FS-inst}] Let $R(p,q,r)$ be the Fintushel-Stern invariant for $\Sigma(p,q,r)$ and suppose that a sequence $\Sigma_i = \Sigma(p_i,q_i,r_i)$ $(i = 1, 2\ldots )$ satisfies that $R(p_i,q_i,r_i ) > 0$ $(i = 1, 2\ldots)$. Then, if $$p_iq_ir_i<p_{i+1}q_{i+1}r_{i+1}, $$ the homology classes $[\Sigma_i]$ $(i = 1, 2\ldots)$ are linearly independent over $\Z$ in $\Theta^3_{\Z}$.
\end{thm*}

In a manner similar to Akbulut's technique, the gauge theoretical result cannot be applied directly; in both cases it is necessary to first construct definite cobordisms from the 2-fold cover $\Sigma_2\left(D_n(T_{p,q})\right)$ to Seifert fibered homology spheres and then apply Furuta's criterion for independence. This approach was used by Hedden-Kirk \cite{hedden-kirk} to establish conditions under which an infinite family of Whitehead doubles of positive torus knots is independent in $\mathcal{C}_\infty$. Nonetheless, their proof involves a complicated computation of bounds for the minimal Chern-Simons invariant of $\Sigma_2\left(D_2\left(T_{p,q}\right)\right)$ and this can be sidestepped by introducing definite cobordisms from $\Sigma_2\left(D_2\left(T_{p,q}\right)\right)$ to Seifert fibred homology 3-spheres. In this article we recover their result and generalize it to include more examples of satellite operations.\\

{\bfseries Outline:} In \autoref{sec::satellites} we offer a brief description of satellite operations and present the important patterns. In \autoref{sec::gauge} we review the theory of $SO(3)$ instantons and the homology cobordism obstruction that derives from it. Then, in \autoref{sec::covers} we explore the topology of the 2-fold covers to later introduce the construction of the relevant cobordisms in \autoref{sec::cobordisms}. Finally, in \autoref{sec::main_result} we prove the main result.\\

{\bfseries Acknowledgements:} The results in this paper originally formed the core of my PhD dissertation. I would like to thank my advisor Prof. Paul Kirk for his patient guidance and continual encouragement throughout my studies. I would also like to thank Prof. Charles Livingston for his careful reading of my dissertation and his helpful comments.

\section{Patterns and Satellite Knots}\label{sec::satellites}
As was mentioned, the main goal of this article is to show independence of families of satellite knots in the smooth concordance group. This is done by considering satellite operations with pattern similar to the Whitehead link and companion a positive torus knot. In this section we describe the patterns of the relevant satellite operations. 

\begin{definition}\label{def:satellite}
Let $B\sqcup A$ be a 2-component link in $S^3$ such that $A$ is an unknot so that $V=S^3\setminus N(A)$ is an unknotted solid torus in $S^3$. For $K$ any knot, consider $h:V\to S^3$ an orientation preserving embedding taking $V$ to a tubular neighborhood of $K$ in such a way that a longitude of $V$ (which is a meridian of $A$) is sent to a longitude of $K$, then $h(B)$ is the untwisted satellite knot with pattern $B\sqcup A$ and companion $K$ and is usually denoted $B(K)$.
\end{definition}

A notable example of a satellite operation is provided by using the Whitehead link (\autoref{Whitehead}) as the pattern of the operation. The knots obtained in this way are called Whitehead doubles. The following figures show the pattern, companion, and satellite whenever we take the pattern $B\sqcup A$ to be the Whitehead link and the companion knot to be the right handed trefoil, $T_{2,3}$.

\begin{figure}[H]
\medskip
\begin{subfigure}[b]{0.3\textwidth}
\centering
\includegraphics[width=0.5\textwidth, height=0.125\textheight]{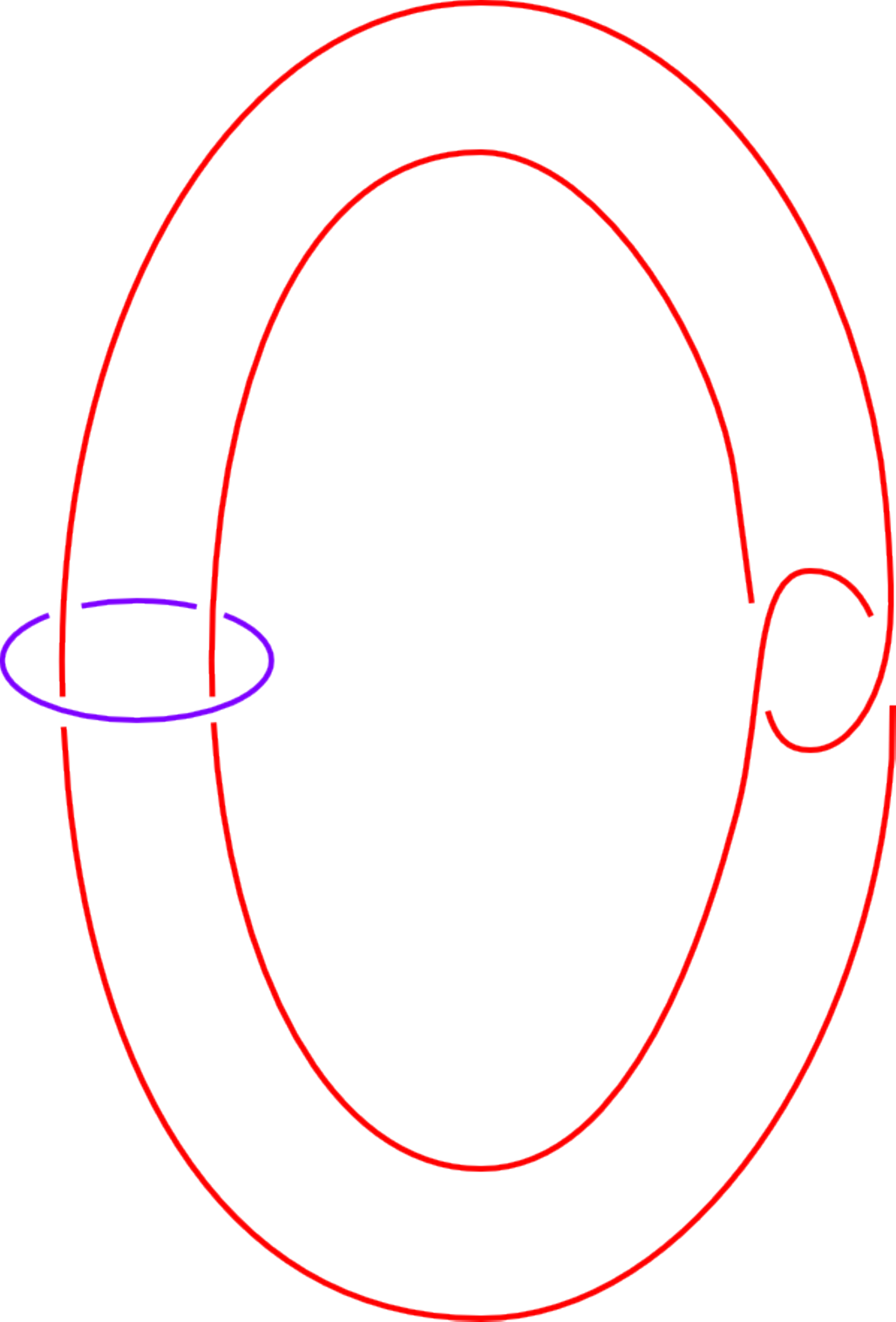}
\caption{The Whitehead Link}
\label{Whitehead}
\end{subfigure}
\begin{subfigure}[b]{0.3\textwidth}
\centering
\includegraphics[width=0.5\textwidth, height=0.125\textheight]{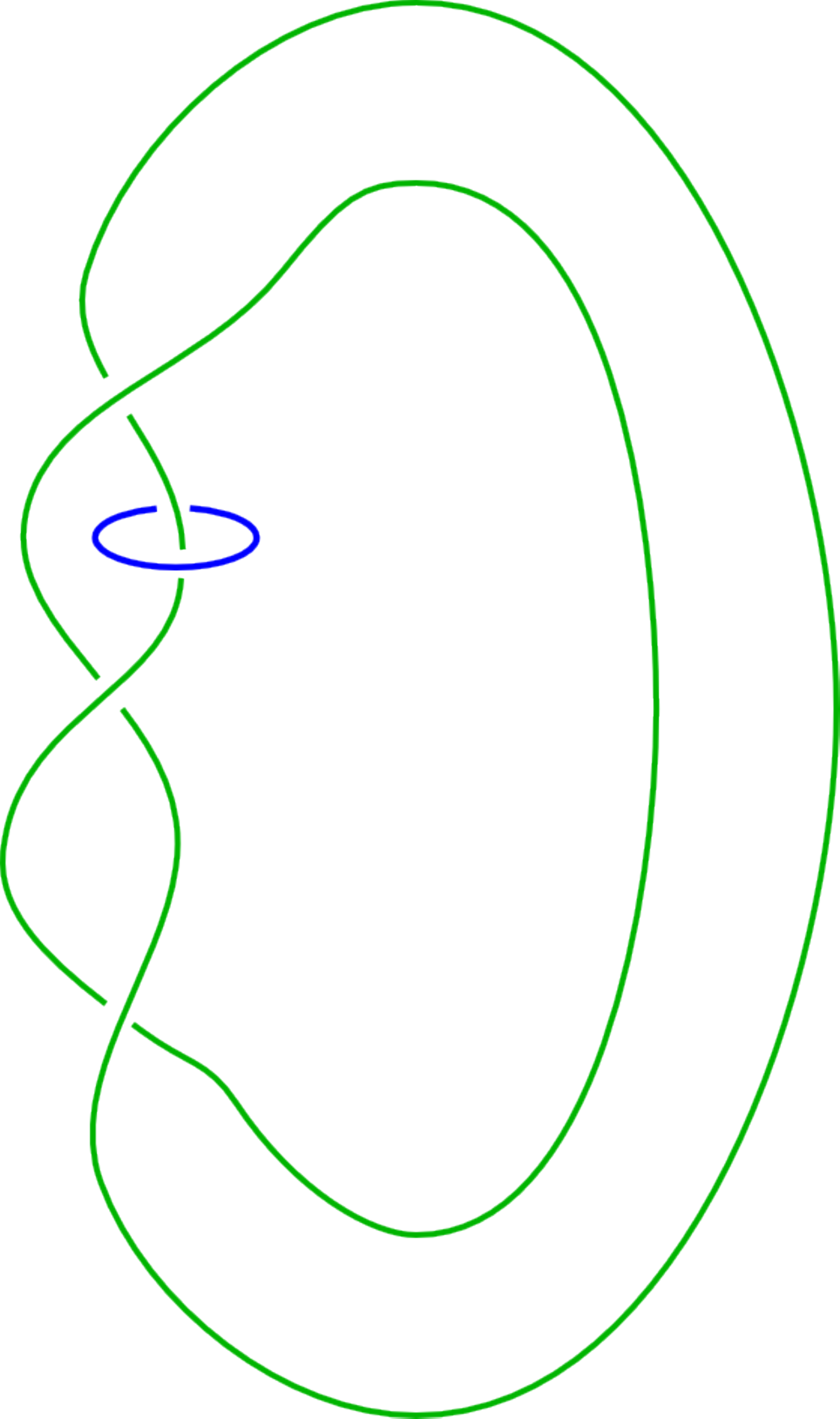}
\caption{Trefoil}
\end{subfigure}
\begin{subfigure}[b]{0.3\textwidth}
\centering
\includegraphics[width=0.5\textwidth, height=0.125\textheight]{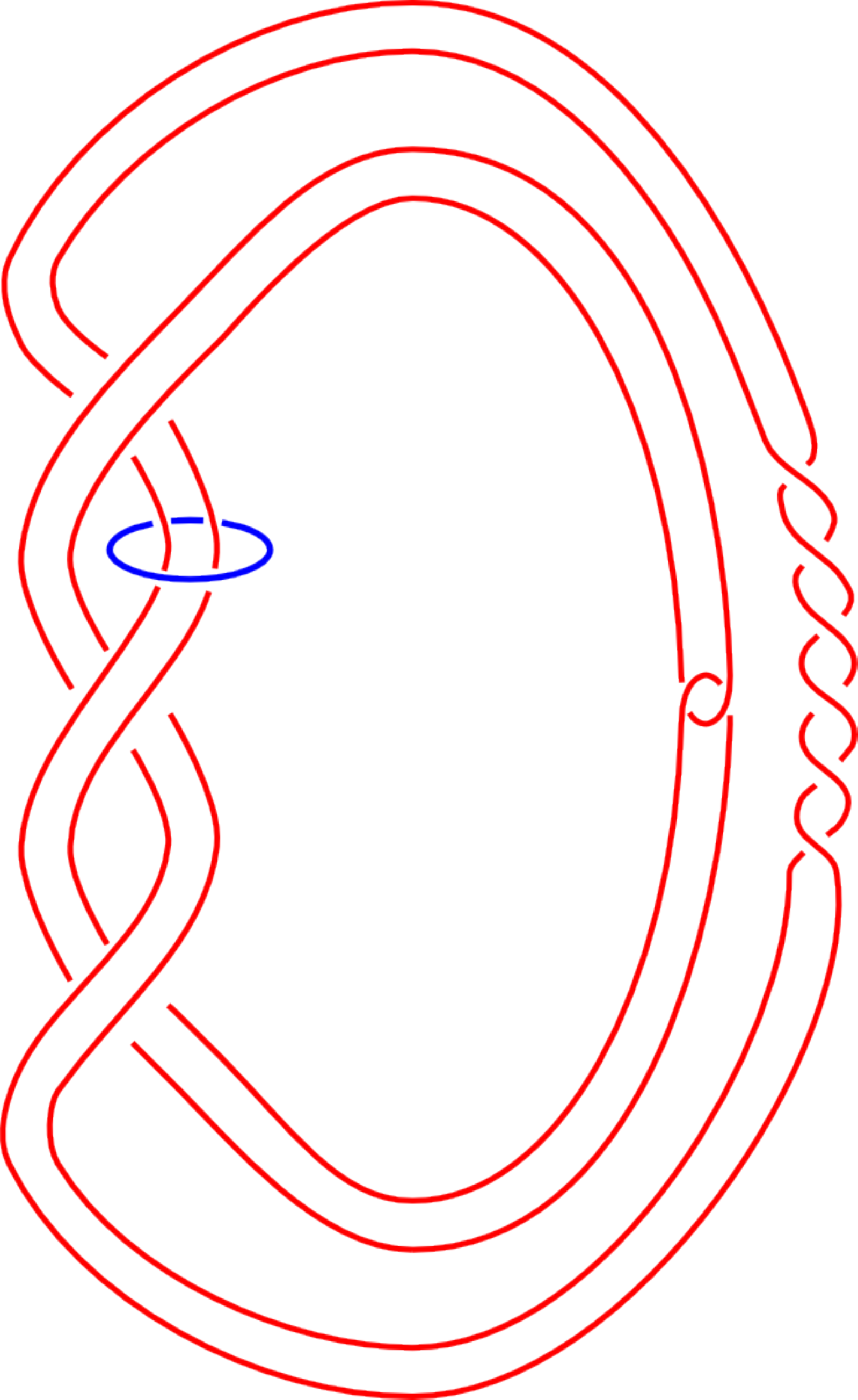}
\caption{Untwisted Whitehead Double of the Trefoil}
\end{subfigure}
\caption{An example of a satellite}
\end{figure}

In greater generality, we can add more twists to the clasp of \autoref{Whitehead} to obtain the patterns included in \autoref{fig::D_n}. These are the Whitehead-like patterns under consideration in the present article.

\begin{figure}[H]
\medskip
\centering
\includegraphics[width=0.225\textwidth,height=0.15\textheight]{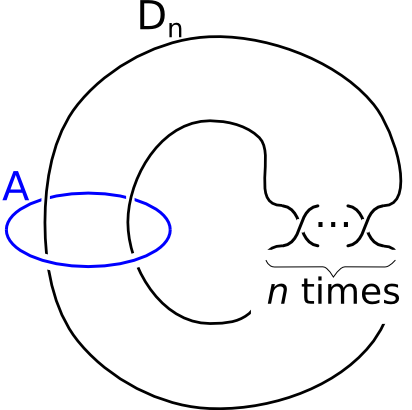}
\medskip
\caption{The Whitehead-like patterns $D_n$.}\label{fig::D_n}
In this figure, $n>0$ denotes the number of positive half twists. Also, since we require $lk(A,D_n)=0$, we will further assume that $n$ is an even integer.
\medskip
\end{figure}

Since the pattern $D_n$, as a knot in $S^3$, is unknotted and $lk(A,D_n)=0$ whenever $n$ is an even integer, the Alexander polynomial of the satellite knot $D_n(K)$ is $\Delta_{D_n(K)}(t)=1$. A theorem of Freedman \cite{freedman, freedman2, freedman-quinn} states that every knot with Alexander polynomial $1$ is topologically slice. This implies that for any companion knot $K$, the satellite $D_n(K)$ is a topologically slice knot. We will later show that whenever $K=T_{p,q}$ with $(p,q)$ a pair of positive and relatively prime integers, the satellite knots $D_n(T_{p,q})$ are not smoothly slice.

\section{Instanton Cobordism Obstruction}\label{sec::gauge}
In this section we survey the theory of instantons on $SO(3)$-bundles developed by  Furuta \cite{furuta} and Fintushel-Stern \cite{FS-pseudofree,FS-inst} in the setting of orbifolds (i.e. manifolds with a special kind of singularities), and recast by Hedden-Kirk \cite{hedden-kirk-cob} in the setting of manifolds with tubular or cylindrical ends. Additionally, in this section we introduce the instanton cobordism obstruction, that is, the way in which the topology of the instanton moduli space obstructs the existence of certain 4--manifolds. \\

Following \cite{furuta,FS-pseudofree,FS-inst,hedden-kirk-cob}, consider the Seifert fibered sphere $\Sigma=\Sigma(p,q,r)$ and the mapping cylinder $\mathcal{W}$ of the Seifert projection $\Sigma\to S^2$. The latter space is a negative definite orbifold with boundary $\Sigma$ and with three singularities, each of which has a neighborhood homeomorphic to a cone on a lens space. To avoid singularities, form a manifold $W=W(p,q,r)$ by removing the aforementioned neighborhoods from the mapping cylinder $\mathcal{W}$, and notice that $$H^2(W;\Z)\cong H^2(W,\Sigma ;\Z)\cong\Z.$$
One of the key components of the theory is that these groups have a preferred generator. Let $e$ be the generator of $H^2(W;\Z)$ and notice that this cohomology class determines an $SO(2)$--vector bundle $\mathcal{L}$ over $W$, which is trivial over $\Sigma$. In addition, if $\varepsilon$ is the trivial real vector bundle of rank 1 over $W$, the bundle $\mathcal{L}\oplus\varepsilon$ is an $SO(3)$--vector bundle over $W$. Then, if $X$ is a 4--manifold with $\Sigma$ as one of its boundary components, one can form $M=X\underset{\Sigma}{\cup} W$ and since $\mathcal{L}\oplus\varepsilon$ is trivial over $\Sigma$, it can be extended trivially to an $SO(3)$--vector bundle $E$ over $M$. \\

For technical reasons originating from analytical considerations, it is necessary to attach to $M$ cylindrical ends isometric to $[0,\infty)\times\partial M$ to form a non-compact manifold $M_\infty$. One then considers the corresponding extension of the bundle $E$ to $M_\infty$ and studies the theory of connections $A$ on $E$ for which the energy integral $$\mathcal{E}(A)=\int\limits_{M_\infty}\text{Tr}(F_A\wedge \ast F_A)<\infty.$$ Here $F_A$ is the curvature of $A$ and $\ast$ is the Hodge star operator. However, one of the subtle variations present in the cylindrical end formulation of the theory of instantons is the presence of limiting connections on $E$ that are determined by the cohomology class $e$. Modulo gauge equivalence, the class $e$ uniquely determines a flat connection $\beta_i$ on the restriction of $L$ to each of the lens spaces in the boundary of $W$. Furthermore, if $\vartheta_i$ is the trivial connection on the restriction of $\varepsilon$ to the $i$-th lens space, we can form $\alpha_i=(\beta_i,\vartheta_i)$ to obtain an $SO(3)$-connection on the restriction of $E$ to the $i$-th lens space. Then, if we choose the trivial $SO(3)$-connection over every other boundary component of $X\cup W$, the tuple $\boldsymbol\alpha=(\alpha_1,\alpha_2,\alpha_3,\theta,\ldots,\theta)$ is the limiting flat connection and $(E,\boldsymbol\alpha)$ is the adapted bundle (in the sense of \cite{donaldson-book}) to be considered.\\

For a positive number $\delta$ and an appropriate weighted Sobolev norm $\| \cdot\|_\delta$, the moduli space $\M=\M_{\delta}(E,\boldsymbol\alpha)$ is the set of gauge equivalence classes of finite weighted norm $SO(3)$-connections $A$ on $E$ that limit to $\boldsymbol\alpha$ and that satisfy the anti-self-dual (ASD) equation $$-F_A=\ast F_A.$$ In other words, $\M$ is the moduli space of instantons over $M_\infty$. Then, perhaps after perturbing either the metric of $M_\infty$ or the anti-self-dual equation, $\M$ can be shown to have the structure of a smooth manifold with some singular points. An in-depth account of the theory of instantons over manifolds with cylindrical ends can be found in \cite{donaldson-book}.\\

In summary, the cohomology class $e$ determines the bundle $(E,\boldsymbol\alpha)$. The next theorem shows that if $X$ is a negative definite 4--manifold, the choice of $e$ also gives information about the topology of the instanton moduli space $\M$ and thus, all the gauge theory over $M_\infty$.

\begin{theorem}\label{moduli} Let $X$ be a negative definite 4-manifold whose boundary consists of the union of some Seifert fibered homology spheres $\Sigma_i=\Sigma(p_i,q_i,k_ip_iq_i-1)$, for $i=1,\ldots,N$. Consider $W=W(p_N,q_N,k_Np_Nq_N-1)$ and form $M=X\union{\Sigma_N}W$. Also, let $E$ be the $SO(3)$-bundle over $M_\infty$ determined by the generator $e$ of $H^2(W;\Z)$.\\
The moduli space $\M$ of finite energy instantons on $E$ is a (possibly non-compact) smooth 1- manifold with boundary and with the following properties:
\begin{enumerate}[label=(\alph*),ref=Theorem \thetheorem (\alph*)]
\item\label{moduli_red} The number of boundary points of $\M$ is given by $C(e)=T/2^{\beta}$ where $T$ is the order of the torsion subgroup of $H_1(X;\Z)$ and $\beta=\text{rank}\left(H_1(X;\Z/2)\right)-\text{rank}\left(H_1(X;\Z)\right)$. 
\item\label{moduli_compact} If $p_iq_i(k_ip_iq_i-1)<p_{i+1}q_{i+1}(k_{i+1}p_{i+1}q_{i+1}-1)$, then $\M$ is compact. 
\end{enumerate}
\end{theorem}

\bigskip
In what follows we offer a broad idea of the proof. For a precise account we refer the reader to \cite{hedden-kirk-cob}.\\

Using the theory of singular bundles over orbifolds, Fintushel-Stern compute the index for the bundle $L\oplus\varepsilon$ over $W(a_1,a_2,a_3)$ and give an explicit formula as $$R(a_1,a_2,a_3)=\frac{2}{a_1a_2a_3}+\sum_{i=1}^3\frac{2}{a_i}\sum_{k=1}^{a_i-1}\cot\left(\frac{\pi ak}{a_i^2}\right)\cot\left(\frac{\pi k}{a_i}\right)\sin^2\left(\frac{\pi k}{a_i}\right).$$
Furthermore, Hedden and Kirk \cite{hedden-kirk-cob} show that whenever $R(a_1,a_2,a_3)$ is positive, it equals the dimension of the moduli space of instantons over the non-compact manifold $M_\infty$ obtained from the augmented manifold $M=X\cup W$, for any 4--manifold $X$. A calculation using the Neumann-Zagier formula \cite{nz} shows that when $p$ and $q$ are relatively prime positive integers and $k\geq1$, the Fintushel-Stern invariant for $\Sigma(p,q,kpq-1)$ is such that $$R(p,q,kpq-1)=1,$$ thus proving that $\M$ is a 1-dimensional space.\\
It can be shown that the boundary points of $\M$ correspond to reducible connections.  Using results found in \cite{Taubes-periodic,hedden-kirk-cob} and some basic algebraic topology one can show that the number of reducible connections is given by $C(e)=T/2^{\beta}$ where $T$ is the order of the torsion subgroup of $H_1(X;\Z)$ and $\beta=\text{rank}\left(H_1(X;\Z/2)\right)-\text{rank}\left(H_1(X;\Z)\right)$ as claimed in \autoref{moduli_red}.\\

Finally, the question about compactness is in fact a question about convergence. To address this question, for any connection $A$ on $E$ that limits to $\boldsymbol\alpha$, consider the integral $$-\frac{1}{8\pi^2}\int\limits_{M_\infty}\text{Tr}(F_A\wedge F_A).$$ The value of this integral can be shown to be independent of the choice of $A$ and thus an invariant of the bundle $(E,\boldsymbol\alpha)$. This invariant is usually denoted by $p_1(E,\boldsymbol\alpha)$ and known as the Pontryagin number of $(E,\boldsymbol\alpha)$. In addition, $p_1(E,\boldsymbol\alpha)$ captures convergence of sequences of connections in $\M$ modulo gauge equivalence. Indeed, Uhlenbeck compactness for non-compact manifolds characterizes lack of convergence in $\M$ as taking one of two different forms:  `Bubbling' and `Breaking'. Bubbling happens when the curvature accumulates near a point inside a compact set in $M_\infty$ and results in a change of the Pontryagin number of the bundle. In fact, by Uhlenbeck's Removable Singularities Theorem \cite{uhlenbeck-rem}, this change comes in multiples of 4 and so, if $p_1(E,\boldsymbol\alpha)$ is less than $4$, bubbling cannot occur. Breaking happens when a region appears in one of the cylindrical ends of $M_\infty$ where the connection looks like an instanton on a tube that limits to a flat connection at either end of the tube. Furuta \cite{furuta} shows that the curvature of the connection at such a region is non-zero and the energy of the connection is greater than or equal to the Chern-Simons invariant of the limiting connections. For ease of notation, for $Y$ a 3--manifold denote by $\tau(Y)$ the minimum of the differences $cs(Y, b) - cs(Y, \boldsymbol\alpha|_Y) \in (0, 4]$, where $b$ ranges over all flat connections on $E|_Y$. So, if $p_1(E,\boldsymbol\alpha)$ is less than $\tau(Y,e)$ for every connected component $Y$ of $\partial M$, breaking cannot occur. In conclusion, if $p_1(E,\boldsymbol\alpha)<4$ and $p_1(E,\boldsymbol\alpha)<\min\{\tau(Y) \mid Y\subseteq \partial M\}$, neither bubbling nor breaking can occur, and thus the previous inequalities constitute a compactness criterion for the moduli space $\M$. Computations of these quantities for the case at hand and proofs of the inequalities will show compactness of $\M$. First, an argument involving the intersection form of $W$ shows that $$p_1(E,\boldsymbol\alpha)=\frac{1}{p_Nq_N(k_Np_Nq_N-1)}<4$$ and can be found in \cite{hedden-kirk-cob}. Further, if $L$ is any of the lens spaces in the boundary of $W$, then its minimum Chern-Simons invariant satisfies $\tau(L)\geq\min\left\{ \frac{1}{p_N},\, \frac{1}{q_N},\,\frac{1}{k_Np_Nq_N-1}\right\}>p_1(E,\boldsymbol\alpha).$  In addition, it is also known \cite{furuta,FS-inst} that $\tau(\Sigma(p,q,kpq-1))= \frac{1}{pq(kpq-1)}.$ Then, the condition $$p_iq_i(k_ip_iq_i-1)<p_{i+1}q_{i+1}(k_{i+1}p_{i+1}q_{i+1}-1)$$ implies $p_1(E,\boldsymbol\alpha)<\min\{\tau(\Sigma_i) \mid i=1,\ldots,N-1\}$ and so, by the compactness criterion previously described, $\M$ is in fact a compact space as asserted in \autoref{moduli_compact}. This completes the sketch of the proof of \autoref{moduli}.\\

To obtain the Instanton Cobordism Obstruction further assume that $H_1(X;\Z/2)=0$. In that case $H_1(X;\Z)$ would be a torsion group with no even torsion and so $\beta$ would be $0$ and $T=|H_1(X;\Z)|$ would be an odd integer. Therefore, $C(e)=|H_1(X;\Z)|$ and the moduli space $\M$ would contain an odd number of reducible connections. However, by \autoref{moduli}, $\M$ is a compact 1-dimensional manifold. Since a compact 1-dimensional manifold cannot have an odd number of boundary components, \autoref{moduli} obstructs the existence of a negative definite 4-manifold satisfying $H_1(X;\Z/2)=0$. The following theorem is a reformulation of \autoref{moduli}, with the additional hypothesis $H_1(X;\Z/2)=0$, expressed in purely topological terms. The only reason for its inclusion is to make the explanation of the result easier to follow by avoiding gauge-theoretical terminology.

\begin{theorem}\label{cobound}Let $p_i,q_i$ be relatively prime integers and $k_i$ a positive integer for $i=1,\ldots,N$. If $\left\{\Sigma_i\right\}_{i=1}^N$ is a family of Seifert fibred homology 3-spheres such that $\Sigma_i=\Sigma(p_i,q_i,k_ip_iq_i-1)$ and satisfying
\begin{equation}\label{criterion}
p_iq_i(k_ip_iq_i-1)<p_{i+1}q_{i+1}(k_{i+1}p_{i+1}q_{i+1}-1),\end{equation}
then no combination of elements in $\left\{\Sigma_i\right\}_{i=1}^N$ cobounds a smooth 4-manifold $X$ with negative definite intersection form and such that $H_1(X;\Z/2)=0$.
\end{theorem}

In summary, the crucial idea is that the topology of the instanton moduli space obstructs the existence of some definite 4--manifolds. Also key is the fact that the cohomology class $e$ and the minimum Chern-Simons invariant of the boundary 3--manifolds provide important information about the topology of the moduli space. Note that the compactness criterion  presented in \autoref{moduli_compact} is precisely the criterion for the independence of a family of satellites of the form $D_n(T_{p,q})$.\\

\section{Topological Description of 2-fold Covers}\label{sec::covers}
A useful method to study the algebraic structure of a group $G$ is to consider homomorphisms $G\to H$ and use information about the algebraic structure of $H$. In the case of the smooth concordance group $\Conc_\infty$ it is common to associate to the concordance class of a knot $K$ the equivalence class of the 2-fold cover of $S^3$ branched over $K$, $\Sigma_2(K)$, in the homology cobordism group of oriented $\Z/2$--homology spheres, $\Theta^3_{\Z/2}$.\\

The following lemma and the comment after it show not only that the (1,2) relation that defines concordance translates into the (3,4) relation that defines homology cobordism, but also that the group operation is preserved by the assignment $K\to \Sigma_2(K)$.
\begin{lemma}[Lemma 2 of \cite{casson-gordon-cob}] \label{2-fold_conc} Let $K \subseteq S^3$ be a knot. Then, 
\begin{enumerate}[label=(\alph*),ref=Theorem \thetheorem (\alph*)]
\item $\Sigma_2(K)$ is a $\Z/2$ homology 3-sphere; that is, $H_*(\Sigma_2(K);\Z/2)\cong H_*(S^3;\Z/2)$. 
\item If $K$ is slice, then $\Sigma_2(K)= \partial Q$, where $Q$ is a $\Z/2$--homology 4--ball, $H_*(Q; \Z/2)\cong H_*(B^4; \Z/2)$.\end{enumerate}\end{lemma}
Moreover, $\Sigma_2(K_1\# K_2)=\Sigma_2(K_1)\#\Sigma_2(K_2)$, where the separating sphere is obtained as the lift of the the embedded 2-sphere in $S^3$ that appears in the definition of $K_1\# K_2$ as the connected sum of pairs $\left( S^3,K_1\right)\# \left( S^3,K_2\right)$. All these observations show that the assignment $K\to \Sigma_2(K)$ is a group homomorphism $$\Sigma_2 : \Conc_\infty \to \Theta^3_{\Z/2}.$$
Therefore, the end result of the present article is in fact a result about independence in $\Theta_{\Z/2}^3$. With all the previous in mind, in this section we include a topological description of $\Sigma_2(D_n(K))$.\\

In \cite{seifert-inv,livingston-melvin} the authors offer a description of the infinite cyclic cover of a satellite knot $B(K)$ in terms of some covers of the companion and pattern knots. Since finite covers may be regarded as quotients of the infinite cyclic covers, their description can be adapted to the case of 2-fold cyclic covers of satellite knots. The branched covers are obtained by compactifying the cyclic cover and attaching to it a solid torus in such a way that a meridian of the solid torus matches with the pre-image of the meridian of the knot in the cyclic cover. In what follows we reproduce without proof the modified version of the description found in \cite{seifert-inv,livingston-melvin}.

\begin{theorem}\label{2-fold_sat} Let $B\sqcup A$ be a pattern link satisfying $lk(A,B)=0$ and $K$ a knot in $S^3$. There are splittings $$\Sigma_2\left(B\right)=V_2\union{} N_2\quad\text{and}\quad \Sigma_2\left(B(K)\right)=W_2\union{} M_2$$ such that
\begin{enumerate}[label=(\alph*),ref=Theorem \thetheorem (\alph*)]
\item The space $N_2$ consists of two copies of $N(A)$ and $M_2$ of two copies of $S^3\setminus N(K)$.
\item If $N^i$ is the $i$-th copy of $N(A)$ in $N_2$ and $X^i$ the $i$-th copy of $S^3\setminus N(K)$ in $M_2$, then $$V_2\cap N^i=T^i\quad\text{and}\quad W_2\cap X^i=U^i$$ where $T^i$ and $U^i$ are 2-tori ($i=1,2$).
\item The embedding $h$ from \autoref{def:satellite} induces a homeomorphism $h_2:V_2\to W_2$.
\item If $q_i$ and $\alpha_i$ are, respectively, the lift of the meridian and longitude of $N$ to $T^i$, then the gluing map of $\Sigma_2\left(B\right)$ identifies $\left(\mu_A\right)_i$ with $q_i$, and $\left(\lambda_A\right)_i$ with $\alpha_i$. Analogously, the gluing map of $\Sigma_2\left(B(K)\right)$ identifies $\left(\lambda_K\right)_i$ with the image of $q_i$ under $h_2$, and $\left(\mu_K\right)_i$ with the image of $\alpha_i$ under $h_2$.
\end{enumerate}
\end{theorem}
In conclusion, there is an isomorphism 
$$\Sigma_2\left(B(K)\right)\cong V_2\union{\phi} 2\left(S^3\setminus N(K)\right),$$ where the gluing map $\phi$ identifies each copy of $\lambda_K$ with $q_i$, and each corresponding copy of $\mu_K$ with the corresponding lift $\alpha_i$. Thus, the 2-fold branched cover of $S^3$ over a satellite knot is determined by the 2-fold cover of a solid torus branched over the pattern $B$, and the curves $\alpha_i$ $(i=1,2)$. The following proposition makes these choices explicit for the patterns presented in \autoref{fig::D_n}.

\begin{proposition}\label{dec_2-fold} Given a knot $K\subseteq S^3$, the 2-fold branched cover $\Sigma_2(D_n(K))$ of $S^3$ branched over $D_n(K)$ has a decomposition $$\Sigma_2(D_n(K))=S^3\setminus N(T_{2,-2n})\union{\varphi} 2\left(S^3\setminus N(K)\right)$$ where $T_{2,-2n}$ is the $(2,-2n)$ torus link with unknotted components $A_1\sqcup A_2$. Additionally, the gluing map $\varphi$ is determined by $$\varphi_*(\mu_K)_i=-n\cdot\mu_{A_i}+\lambda_{A_i}\text{ and }\varphi_*(\lambda_K)_i=\mu_{A_i},$$ where $\mu_{A_i},\,\lambda_{A_i}$ ($i=1,2$) denote the standard meridian-longitude pairs for the components of the link $T_{2,-2n}=A_1\sqcup A_2$, and $(\mu_K)_i,\,(\lambda_K)_i$ ($i=1,2$) denote the standard meridian-longitude pairs for $K$.
\end{proposition}

\begin{proof}  By \autoref{2-fold_sat}, to obtain a description of $\Sigma_2(D_n(K))$ it is enough to understand $V_2$, the 2-fold cover of $V=S^3\setminus N(A)$ branched over $D_n$, and $\alpha_i$ ($i=1,2$), the lifts of $\lambda_A$ to $V_2$. Since the longitude $\lambda_A$ of $A$ is the $0$--framing of $A$, it suffices to consider the framed knot $(A,0)$ and its framed lifts $(A_i,f_i)$ ($i=1,2$). Indeed, if $A_1\sqcup A_2$ is the lift of $A$ to $\Sigma_2(D_n)$, then $$V_2=\Sigma_2(D_n)\setminus N(A_1\sqcup A_2)\cong S^3\setminus N(A_1\sqcup A_2),$$ and $\alpha_i$ is the $(f_i,1)$ curve in $\partial N(A_i)$ ($i=1,2$). Therefore, the proof amounts to getting a description of the cover $\Sigma_2(D_n)$, which, since $D_n$ is trivial in $S^3$, is simply $S^3$. This uses the surgery description of the pattern $D_n\sqcup A$ and is included as \autoref{surgery_D_n}.
\end{proof}

\begin{figure}[h]
\medskip
\centering

\footnotesize
\newcolumntype{V}{>{\centering\arraybackslash\hspace{0pt}}p{0.24\textwidth}}
\begin{subfigure}[b]{\textwidth}\centering
\begin{tabular}{V<{\hspace{1cm}}V<{\hspace{1cm}}V}
\includegraphics[width=0.175\textwidth]{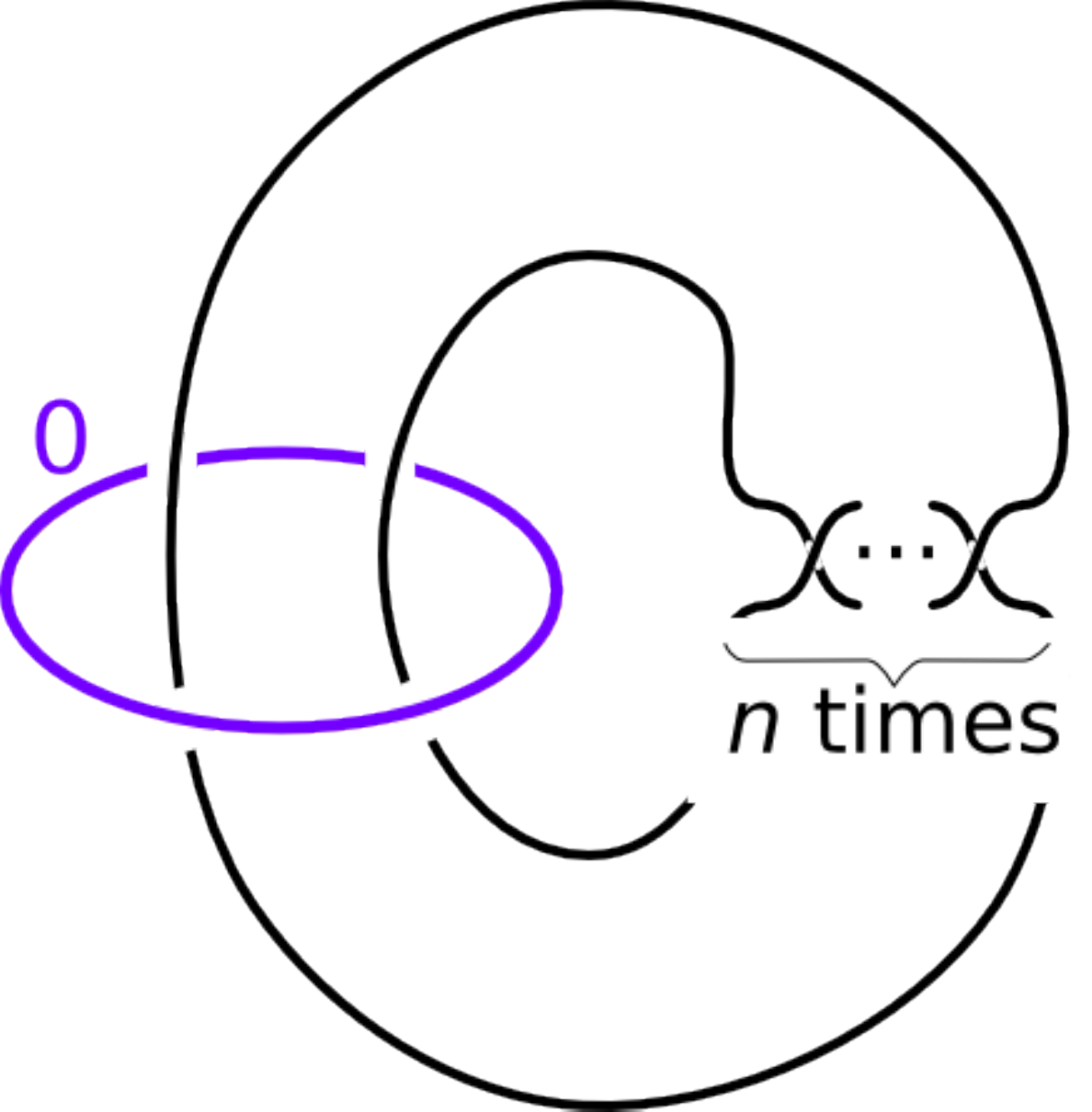}&
\includegraphics[width=0.175\textwidth]{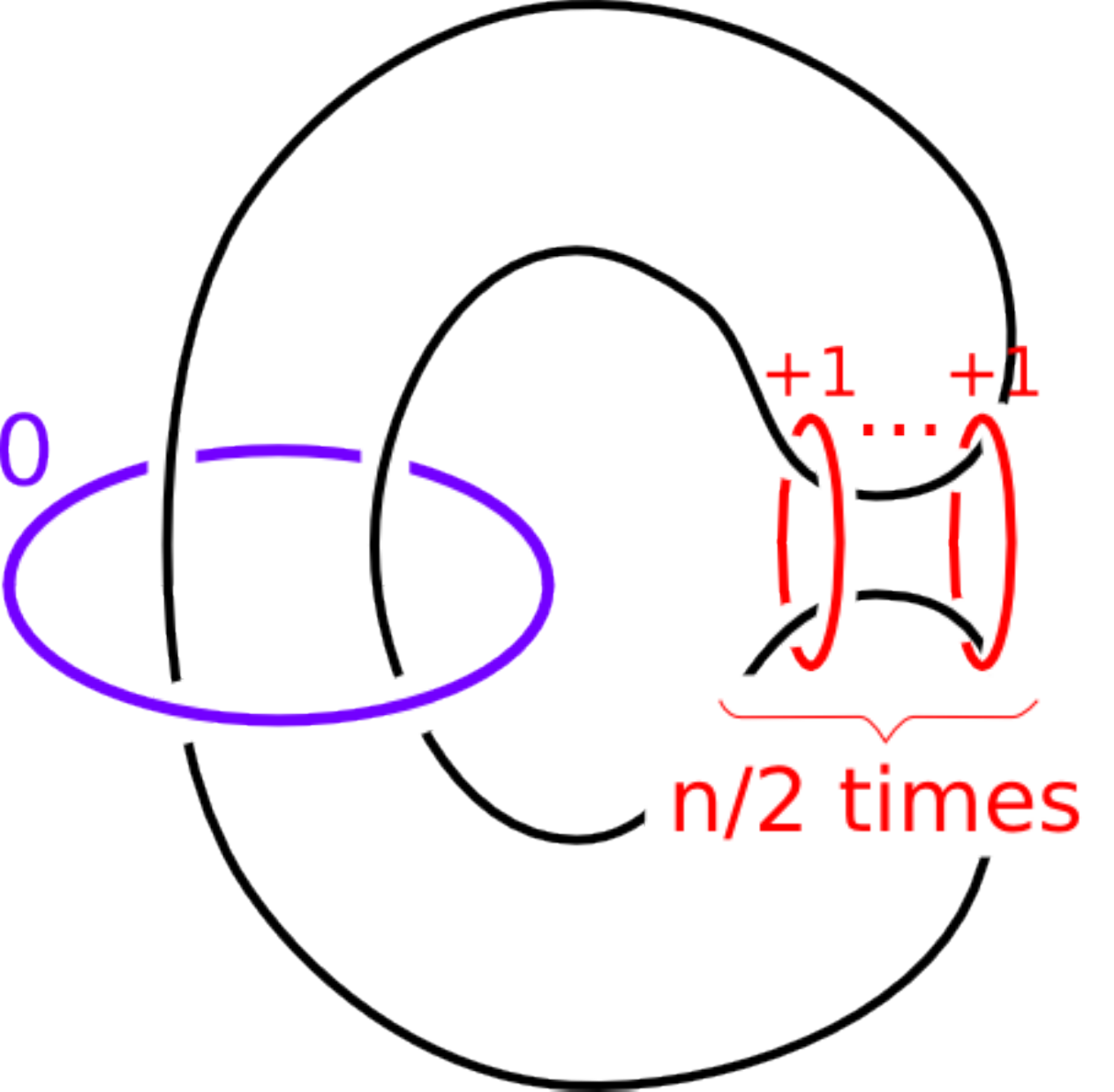}&
\includegraphics[width=0.15\textwidth, height=3.5cm]{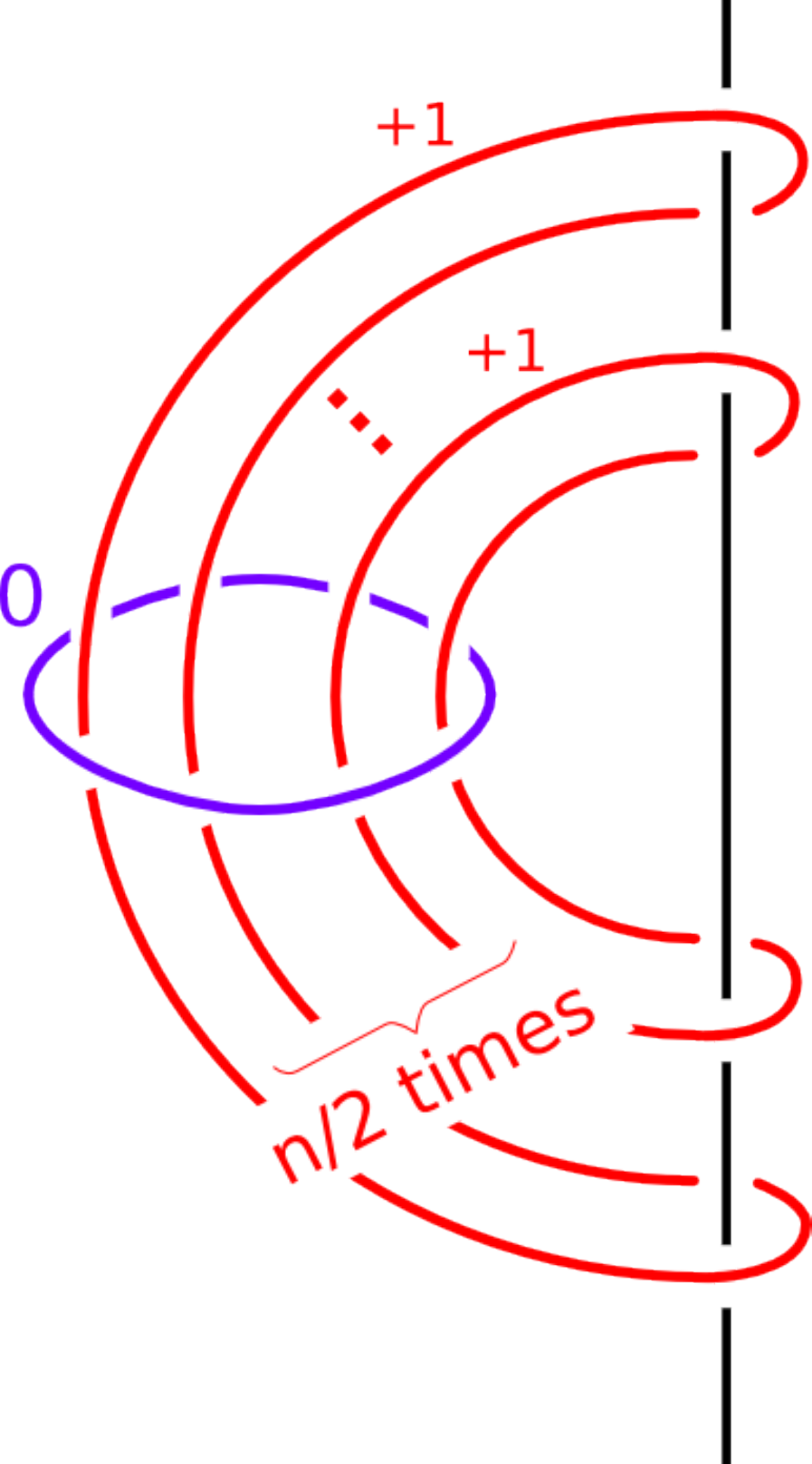}\\
The pattern $D_n$ and the pair $(A,0)$.&
Surgery description of the pattern $D_n$ as a subset of $V$.&
An isotopy of the previous diagram.
\end{tabular}
\caption{Surgery description of $D_n$ as a subspace of $V$.}
\end{subfigure}

\newcolumntype{W}{>{\centering\arraybackslash\hspace{0pt}}p{0.4\textwidth}}
\begin{subfigure}[b]{\textwidth}\centering
\begin{tabular}{W<{\hspace{1cm}}W}
\includegraphics[width=0.2\textwidth]{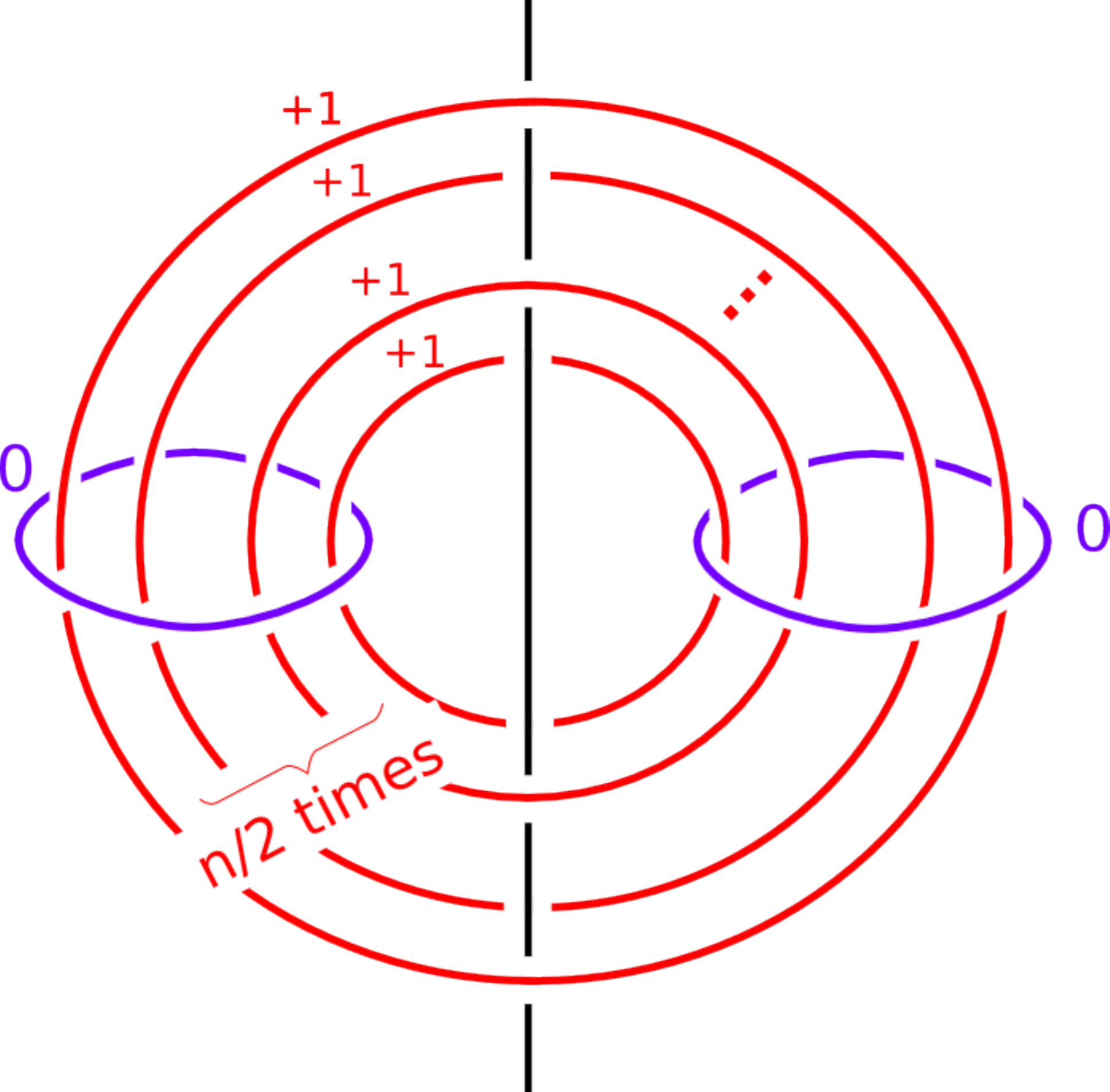}&
\includegraphics[width=0.2\textwidth, height=3cm]{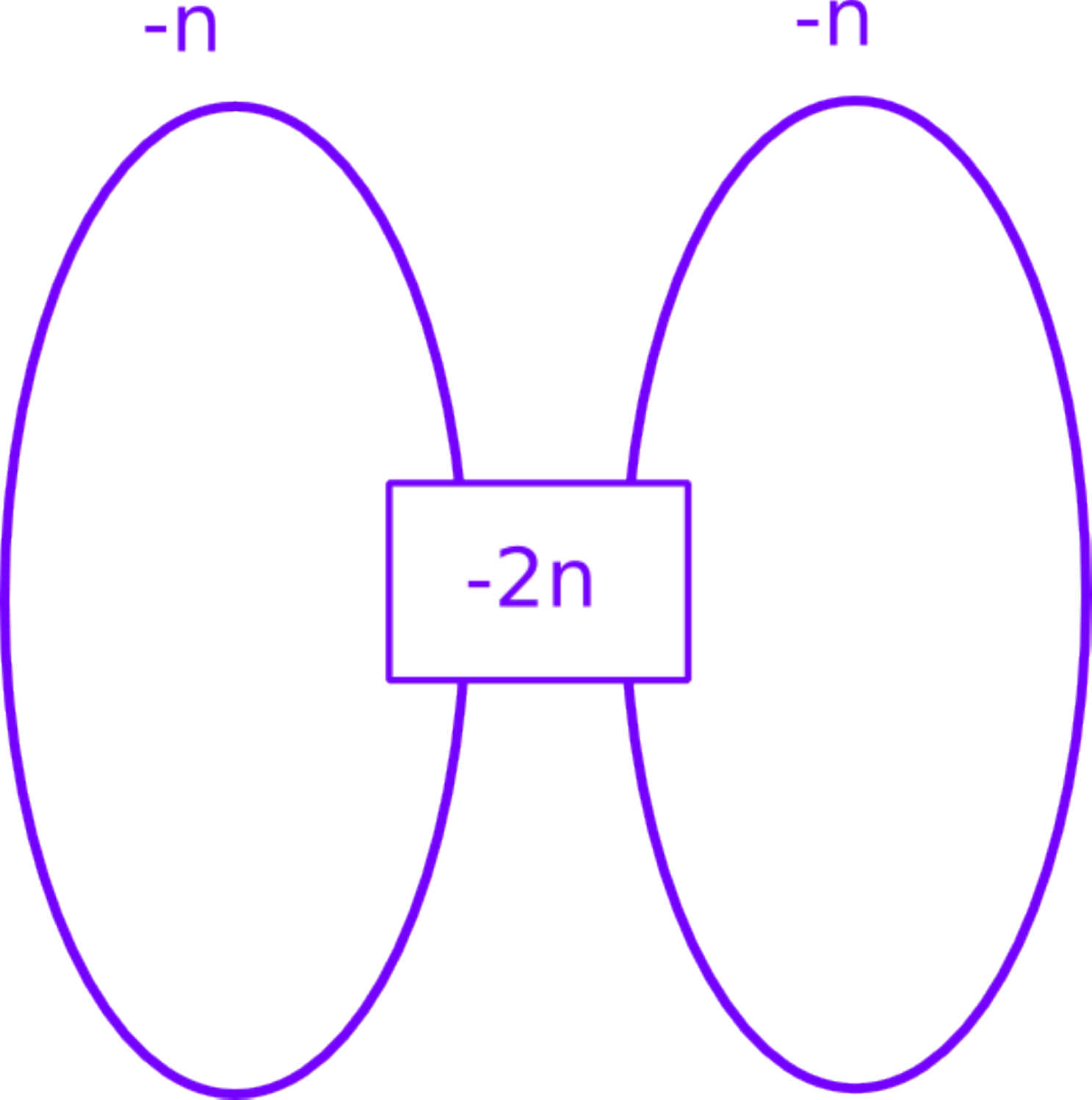}\\
Surgery diagram of the 2-fold cover of $V$ branched over $D_n$.\vfill&
Performing the surgeries one obtains the 2-fold cover of $V$ branched over $D_n$ and the lifts of $(A,0)$. The box represents half twists.
\end{tabular}
\caption{Description of the 2-fold cover of $V$ branched over $D_n$.}
\end{subfigure}
\medskip
\caption{Surgery description of $D_n$ and 2-fold cover of $D_n(K)$}\label{surgery_D_n}
\end{figure}

\section{Definite Cobordisms}\label{sec::cobordisms}
As was mentioned earlier, the main result will be obtained in terms of the instanton cobordism obstruction presented in \autoref{cobound} for a collection of Seifert fibered homology 3--spheres to cobound a negative definite 4--manifold. The issue here is that the 3--manifold $\Sigma_2(D_n(T_{p,q}))$ is not Seifert fibered. However, this obstacle can be overcome by introducing definite cobordisms with (unoriented) boundary $\Sigma_2(D_n(T_{p,q}))$ and some Seifert fibered spaces. In this section the construction of the above cobordisms is included. We start by recalling the precise definition of an oriented cobordism.

\begin{definition}\label{def_cob} Closed, oriented 3-manifolds $Y_0$ and $Y_1$ are oriented cobordant if there exists a compact, oriented 4-manifold $W$ with oriented boundary $$\partial W=-Y_0\sqcup Y_1\footnote{The convention for the induced orientation of a 3-manifold $M$ that is a boundary is to use the outward normal as the first element in an ordered basis for $H_3(M;\Z)$.}.$$
The manifold $W$ is called a cobordism from $Y_0$ to $Y_1$ with $Y_0$ referred to as the incoming boundary component and $Y_1$ the outcoming boundary component. Moreover, if $W$ is positive (negative) definite, then $W$ is called a positive (negative) definite cobordism.
\end{definition}

The following theorem indroduces the sought after cobordisms:
\begin{theorem}\label{cobordisms}Let $(p,q)$ be relatively prime positive integers and $n>0$ an even integer. If $\Sigma_2(D_n(T_{p,q}))$ is the 2-fold cover of $S^3$ branched over the satellite knot $D_n(T_{p,q})$, then there exist:
\begin{enumerate}[label=(\alph*),ref=Theorem \thetheorem (\alph*)]
\item\label{Z} A negative definite cobordism $Z(n,p,q)$ from $\Sigma_2(D_n(T_{p,q}))$ to $-\Sigma(p,q,npq-1)$. 
\item\label{R} A negative definite cobordism $R(n,p,q)$ from $\Sigma_2(D_n(T_{p,q}))$ to the empty manifold.
\item\label{P} A positive definite cobordism $P(n,p,q)$ from $\Sigma_2(D_n(T_{p,q}))$ to ${-\Sigma(p,q,2npq-1)\sqcup-\Sigma(p,q,2npq-1)}$.
\end{enumerate}
Moreover, these cobordisms have trivial first homology group $H_1(\--;\Z)$.
\end{theorem}

By the handle decomposition theorem, every cobordism with incoming boundary component $Y$ is obtained by attaching handles to $I\times Y$. Requiring the cobordism to be oriented is equivalent to requiring the attaching maps of the 1-handles to preserve orientations. In the case being considered, all the cobordisms will be obtained by attaching 2--handles to the 4--manifold $I\times\Sigma_2(D_n(T_{p,q}))$ along framed knots in $\{1\}\times\Sigma_2(D_n(T_{p,q}))$. To that end, we first recall the precise definition of framings to later compute the relevant ones.
\begin{definition}\label{framing} Let $J$ be a knot in a $\Z$--homology sphere $Y$ and $N(J)$ a tubular neighborhood of $J$ in $Y$. A framing of $J$ is a choice of a simple closed curve $J'$ in the boundary $N(J)$ that wraps once around $J$ in the longitudinal direction. Similarly, the framing coefficient of $J$ is the oriented intersection number of $J'$ and any Seifert surface for $J$ in $Y$.\end{definition}

With the definition of framing at hand, we start with the construction of the cobordism $Z(n,p,q)$.
\begin{proof}[Proof of \autoref{Z}]
Any torus knot $T_{p,q}$ with $(p,q)$ relatively prime and positive integers admits a planar diagram with only positive crossings. This implies that $T_{p,q}$ can be unknotted by a sequence of positive-to-negative crossing changes in such a way that the $i$-th crossing change is obtained by performing $-1$ surgery on $S^3$ along a trivial knot $\gamma_i$ that lies in the complement of $T_{p,q}$ and encloses the crossing. Then, if $c$ is the number of crossings changed and $L=\gamma_i\sqcup\ldots\sqcup\gamma_c$, there exists an isomorphism \begin{equation}\label{iso}\psi:S^3_{-1}(L)\to S^3\end{equation} that identifies the restriction of $T_{p,q}$ to the complement of $L$ with the unknot. Next, notice that since $L$ is contained in $S^3\setminus N(T_{p,q})$, it can be regarded as a subset of $\Sigma_n=\Sigma_2(D_n(T_{p,q}))$. Thus, one can form a 4--manifold $Z$ by attaching 2--handles to $I\times\Sigma_n$ along the framed link $(L,-1)$. Specifically, if $\mathbf{h}_i$ is a 4-dimensional 2-handle, 
$$Z=\left(I\times\Sigma_n\right)\union{L}\left(\mathbf{h}_1\sqcup\ldots\sqcup\mathbf{h}_c\right).$$
It is then a matter of routine to check that the incoming boundary component of $Z$ is the manifold $\Sigma_n$ and its outcoming boundary component, $Y$, is the result of surgery on $\Sigma_n$ along the framed link $(L,-1)$. In what follows, we will first obtain a description of $Y$ as surgery and then we will show that $Z$ is a negative definite manifold.\\

First, using the description of $\Sigma_n$ included in \autoref{dec_2-fold}, $Y$ can be seen to split as the union of $ \left(S^3\setminus N(T_{p,q})\right) \union{\varphi_1} \left(S^3\setminus N(T_{2,-2n})\right)$ and the result of surgery on $S^3\setminus N(T_{p,q})$ along the framed link $(L,-1)$. The restriction of the isomorphism $\psi$ from \autoref{iso} to the latter space shows that surgery on $S^3\setminus N(T_{p,q})$ along the framed link $(L,-1)$ is isomorphic to the unknot complement and therefore isomorphic to a standard solid torus $D^2\times S^1$. Furthermore, choosing $\gamma_i$ to have linking number $0$ with the knot $T_{p,q}$ guarantees that the Seifert longitude of $T_{p,q}$ gets sent to the Seifert longitude of the unknot, and thus to a meridional curve $\partial D^2\times\{\text{pt.}\}$ of $D^2\times S^1$. The aforementioned choice also guarantees that the meridian of $T_{p,q}$ gets sent to the longitudinal curve $\{\text{pt.}\}\times S^1$ of the solid torus $D^2\times S^1$. In other words, if $h$ is the isomorphism between surgery on $S^3\setminus N(T_{p,q})$ and the standard solid torus $D^2\times S^1$, there is an isomorphism $$Y\cong\left(S^3\setminus N(T_{p,q})\right)\union{\varphi_1}\left(S^3\setminus N(T_{2,-2n})\right) \union{\varphi_2\circ h} D^2\times S^1.$$
To simplify notation call $A_1,A_2$ the components of the link $T_{2,-2n}$ and let  $$X=\left(S^3\setminus N(T_{2,-2n})\right)\union{\varphi_2\circ h}D^2\times S^1=\left(S^3\setminus N(A_1\sqcup A_2)\right)\union{\varphi_2\circ h}D^2\times S^1.$$ 
Notice that since the gluing map $\varphi_2\circ h:\partial D^2\times S^1\to \partial N(A_1\sqcup A_2)$ satisfies $$\left(\varphi_2\circ h\right)_*(\left[S^1\right])=(\varphi_2)_*(\mu_K)=-n\mu_{A_2}+\lambda_{A_2}\quad\text{ and }\quad \left(\varphi_2\circ h\right)_*(\left[\partial D^2\right])=(\varphi_2)_*(\lambda_K)=\mu_{A_2},$$ it extends to the interior of $D^2\times S^1$. This implies that $X$ is the result of filling the space left by $N(A_2)$ in $S^3$ with a solid torus in a way that makes $X$ isomorphic to $S^3\setminus N(A_1)$. Then, since $A_1$ is unknotted, $X$ is isomorphic to a standard solid torus and thus $Y$ is isomorphic to the union of $S^3\setminus N(T_{p,q})$ and a solid torus. In other words, $Y$ is the result of performing surgery on $S^3$ along $T_{p,q}$. To make explicit the coefficient of the surgery, recall that $$(\varphi_1)_*(\mu_K)=-n\mu_{A_1}+\lambda_{A_1}\quad\text{ and }\quad(\varphi_1)_*(\lambda_K)=\mu_{A_1}.$$ Then, since $\lambda_{A_1}$ is identified with the meridian $\partial D^2$ and $\mu_{A_1}$ with the longitude $S^1$, simple arithmetic shows that $$(\varphi_1)_*(\mu_K+n\lambda_K)=\left[\partial D^2\right],$$
thus showing that the surgery coefficient is $1/n$. Finally, since for $p,q,n>0$, the result of $1/n$ surgery on $S^3$ along the torus knot $T_{p,q}$ is diffeomorphic to the Seifert fibred homology sphere ${-\Sigma(p,q,npq-1)}$ \cite[Proposition 3.1]{moser}, the outcoming boundary component of $Z$ is $-\Sigma(p,q,npq-1)$ as sought.\\

As for definiteness, since $\Sigma_n$ is a homology sphere, the second homology group $H_2(Z;\Z)$ admits a basis determined by the 2--handles. In addition, the matrix representation of the intersection form of $Z$ in terms of such basis is given by the linking matrix of the framed link $(L,-1)$. This, in turn, can be seen to be the matrix $-I_c$, where $I_c$ is the $c\times c$ identity matrix. We thus see that $Z$ is negative definite as sought.\end{proof}

The remaining statements \autoref{R}, and \autoref{P} will be obtained as a corollary to the following theorem.
\begin{theorem}\label{mine} Let $K$ be any knot and $\Sigma_n$ the 2--fold cover of $S^3$ branched over $D_n(K)$. Then, there exist 4--manifolds $P_n(K)$ and $R_n(K)$ such that
\begin{enumerate}[label=(\alph*),ref=Theorem \thetheorem (\alph*)]
\item $P_n(K)$ is a positive definite cobordism from $\Sigma_n$ to $S^3_{\frac{1}{2n}}(K)\#S^3_{\frac{1}{2n}}(K)$.
\item $R_n(K)$ is a negative definite cobordism from $\Sigma_n$ to $S^3$ \end{enumerate}
\end{theorem}

The cobordisms will be constructed explicitly from $I\times\Sigma_n$ by attaching some 2-handles to it along framed knots in $\Sigma_n$. Specifically, the attachment will take place along the links $\boldsymbol\gamma^\pm=\gamma^\pm_1\sqcup\ldots\sqcup\gamma^\pm_n$ shown in \autoref{my_cob} and will be completely determined after establishing the appropriate framing  and framing coefficient of the link components. Notice that since $\boldsymbol\gamma^\pm$ is completely contained in $S^3\setminus N(T_{2,-2n})$, any tubular neighborhood $N(\gamma^\pm_i)$ in $S^3$ small enough to be completely contained in $S^3\setminus N(T_{2,-2n})$ is also a tubular neighborhood of $\gamma_i^\pm$ in $\Sigma_n$. \autoref{framing} and the previous statement show that there is no difference between framings of $\boldsymbol\gamma^\pm$ in $S^3$ and $\Sigma_n$. To see that the same holds for framing coefficients we need to analyze the Seifert surfaces for $\gamma_i^\pm$ in both $S^3$ and $\Sigma_n$. First, since $\gamma_i^\pm$ is an unknot in $S^3$, any embedded 2-disk in $S^3$ bounding $\gamma_i^\pm$ is a Seifert surface for $\gamma_i^\pm$ in $S^3$. Call such disk $D_i$ and choose it to be disjoint from every other component of $\boldsymbol\gamma^\pm$. Notice also that each curve $\gamma_i^\pm$ encloses a crossing of $T_{2,-2n}$ in such a way that $D_i$ intersects the boundary of $N(T_{2,-2n})$ in two disjoint curves, one homologous to $-\mu_{A_1}$ and the other to $-\mu_{A_2}$ (See \autoref{gamma}). Next, to obtain a Seifert surface $S_i$ for $\gamma_i^\pm$ in $\Sigma_n$, let $F_j$ be a Seifert surface for $K$ in $S^3$ contained in the $j$-th copy of $S^3\setminus N(K)$ in $\Sigma_n$ and recall that the gluing map $\varphi$ from \autoref{dec_2-fold} identifies $\mu_{A_j}$ with a longitude of $K$ in the $j$-th copy of $S^3\setminus N(K)\subseteq \Sigma$. The surface $F_j$ can then be glued to $D_i$ along $\varphi_j$ and so we can form \begin{equation}\label{seifert_sigma}S_i=D_i\cap S^3\setminus N(T_{2,-2n}) \union{\varphi} (F_1\sqcup F_2).\end{equation} Hence, if $\beta_i$ is any framing of $\gamma_i$, its framing coefficient in $S^3$ is given by the number of points in $\beta_i\cap D_i$ counted with sign and its framing coefficient in $\Sigma_n$ is given by the number of points in $\beta_i\cap S_i$ counted with sign. Since any choice of $\beta_i$ is contained in the interior of $S^3\setminus N(T_{2,-2n})$, it is disjoint from each copy of $S^3\setminus N(K)$ that appears in the description of $\Sigma_n$. Thus, $\beta_i$ is disjoint from both $F_1$ and $F_2$ and so $$\beta_i\cap S_i=\beta_i \cap \left(D_i\cap S^3\setminus N(T_{2,-2n})\right)=\beta_i\cap D_i.$$ This shows that the framing coefficient of $\gamma_i^\pm$ in both $S^3$ and $\Sigma_n$ agree.\\

\begin{figure}[H]
\medskip
\centering
\begin{subfigure}[b]{0.475\textwidth}
\centering
\def\svgwidth{\textwidth}
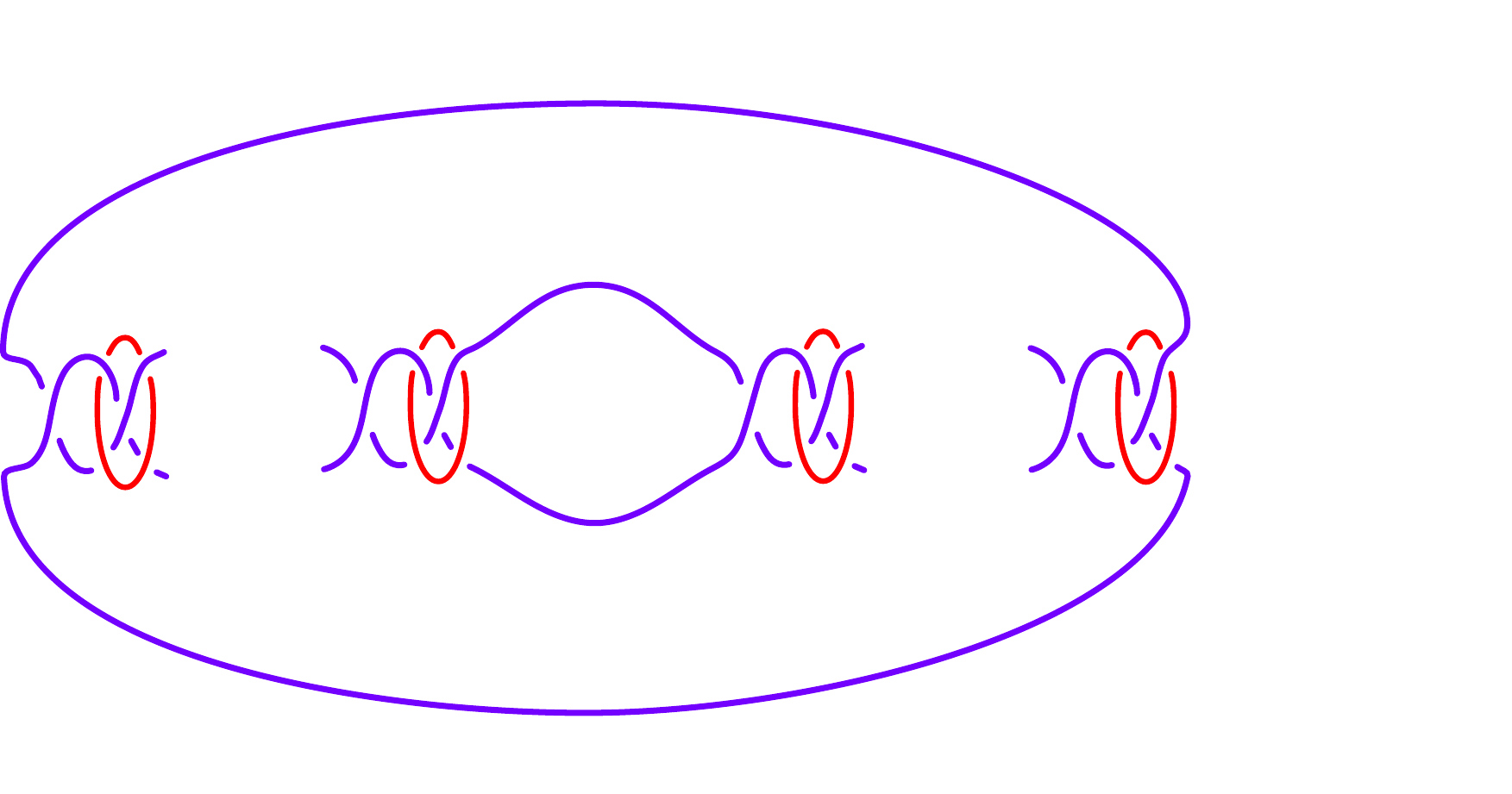
\caption{$\Sigma_n$ and the link $\boldsymbol\gamma^+$.}
\end{subfigure}
\begin{subfigure}[b]{0.475\textwidth}
\centering
\def\svgwidth{\textwidth}
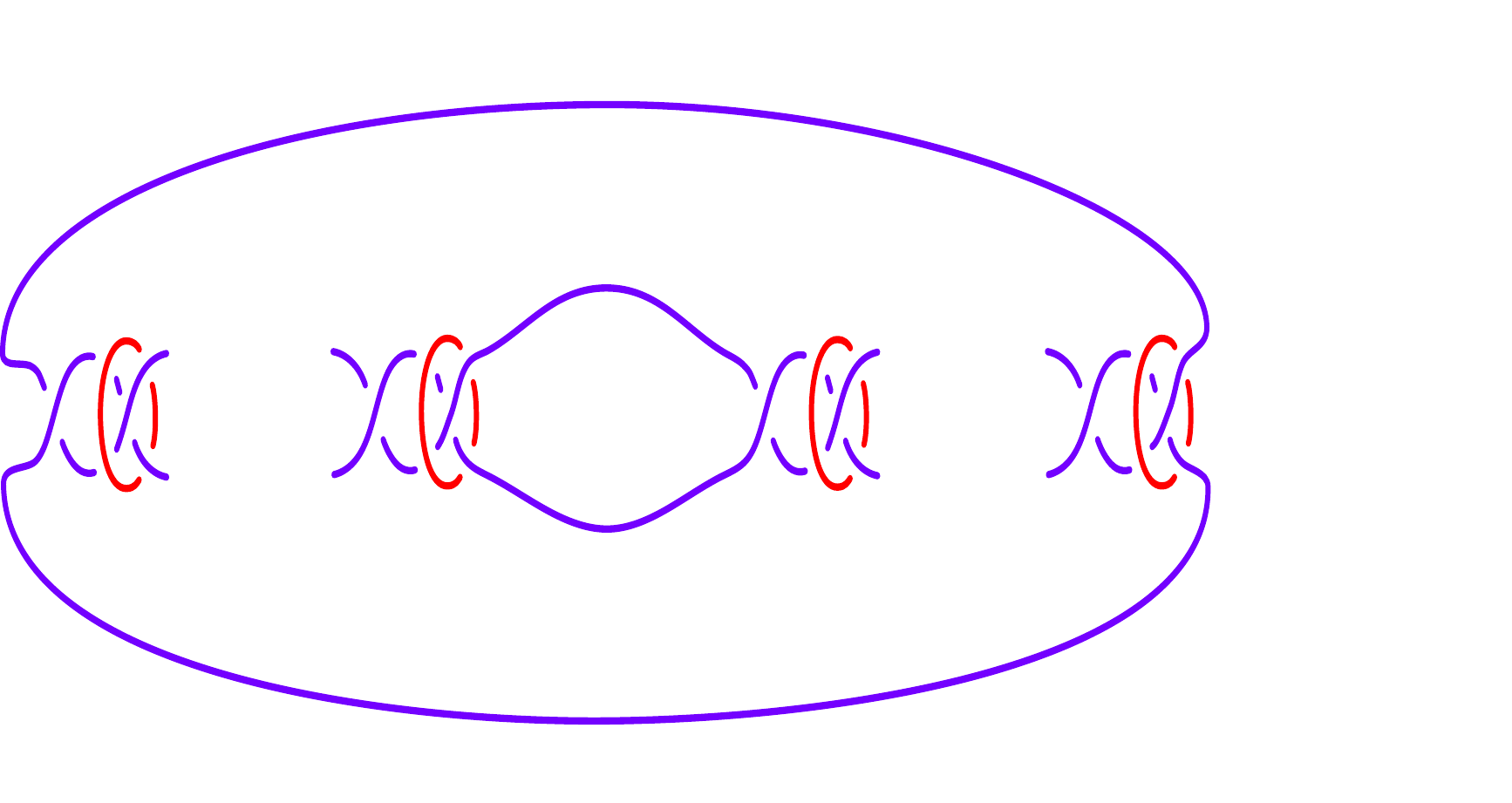
\caption{$\Sigma_n$ and the link $\boldsymbol\gamma^-$.}
\end{subfigure}
\caption{Descriptions of $P_n(K)$ and $R_n(K)$}\label{my_cob}
\end{figure}

So, let $\boldsymbol\gamma^\pm=\gamma^\pm_1\sqcup\ldots\sqcup\gamma^\pm_n$ with the framings as shown in \autoref{my_cob}, and form 
$$P_n(K)=\left(I\times \Sigma_n\right)\union{\boldsymbol\gamma^+}\left(\mathbf{h}_1\sqcup\ldots\sqcup\mathbf{h}_n\right)\qquad\text{and}\qquad R_n(K)=\left(I\times \Sigma_n\right)\union{\boldsymbol\gamma^-}\left(\mathbf{h}_1\sqcup\ldots\sqcup\mathbf{h}_n\right).$$
These two 4--manifolds are the sought after cobordisms, as will be established next.

\begin{proof}[Proof of \autoref{mine}] The boundary of $P_n(K)$ is the disjoint union of $-\Sigma_n$ and $M^+$, the result of surgery on $\Sigma_n$ along the framed link $\boldsymbol\gamma^+$. Analogously, the boundary of $R_n(K)$ is the disjoint union of $-\Sigma_n$ and $M^-$, the result of surgery on $\Sigma_n$ along the framed link $\boldsymbol\gamma^-$. Then, since $\boldsymbol\gamma^\pm$ is a link in $S^3\setminus N(T_{2,-2n})$, the space $M^\pm$ can be expressed as the union of two disjoint copies of $S^3\setminus N(K)$ and surgery on $S^3\setminus N(T_{2,-2n})$ along the framed link $\boldsymbol\gamma^\pm$. The latter manifold can be better understood by first performing the surgery on $S^3$ and then examining the effect such surgery has on $S^3\setminus N(T_{2,-2n})$.\\

Since the surgery is done along unknots with framing $\pm 1$, the result is a space isomorphic to $S^3$. Also, notice that every component of $\boldsymbol\gamma^\pm$ encloses a crossing of the link $T_{2,-2n}$. Then, it is well-known that surgery on $S^3$ along $\boldsymbol\gamma^\pm$ can be interpreted as a sequence of $n$ crossing changes on the link $T_{2,-2n}$ that unlink its components. In other words, there is an isomorphism $$\psi^\pm:\: S^3_{\pm 1}\left(\boldsymbol\gamma^\pm\right)\to S^3$$ that sends the restriction of $T_{2,-2n}$ to the complement of $\boldsymbol\gamma^\pm$, to the 2-component unlink $\mathbf{U}=U_1\sqcup U_2$. Thus, after restricting, $\psi^\pm$ gives us an isomorphism between surgery on $S^3\setminus N(T_{2,-2n})$ along $\boldsymbol\gamma^\pm$ and $S^3\setminus N(\mathbf{U})$. The previous shows that $$M^\pm\cong\left(S^3\setminus N(\mathbf{U})\right)\union{\psi\circ\varphi} 2\left(S^3\setminus N(K)\right).$$ Furthermore, since $\mathbf{U}$ is a 2-component unlink, there exists a 2--sphere $S^2$ that separates $S^3\setminus N(\mathbf{U})$ into $S^3\setminus N(U_1)\#S^3\setminus N(U_1)\cong D^2\times S^1 \# D^2\times S^1$. Then, the same sphere decomposes $M^\pm$ as \begin{equation}\label{outer} M^\pm\cong\left(\left(D^2\times S^1\right)\union{h^\pm_1} \left(S^3\setminus N(K)\right)\right)\#\left(\left(D^2\times S^1\right)\union{h^\pm_2} \left(S^3\setminus N(K)\right)\right).\end{equation}
For simplicity in notation, set $X^\pm=\left(D^2\times S^1\right)\union{h^\pm} \left(S^3\setminus N(K)\right)$ and notice that, being the union of the complement of $K$ and a solid torus, $X^\pm$ is surgery on $S^3$ along $K$. The coefficient of the surgery is given by the homology class of the curve that maps to the meridian $\partial D^2\times\{pt.\}$ of $D^2\times S^1$ under the gluing map $h^\pm$, and so it is important to understand $h^\pm$. This can be done by analyzing the identifications that took place to get \autoref{outer}, and the effect they have on $\mu_k$ and $\lambda_K$. With that in mind, let $\{\mu_{A_i},\lambda_{A_i}\}$ be the meridian-longitude pair of the component $A_i$ of $T_{2,-2n}$, and let $\{\mu_{U_i},\lambda_{U_i}\}$ be the meridian-longitude pair of the component $U_i$ of $\mathbf{U}$. Also, recall that $\varphi$ is such that $$(\varphi)_*(\mu_K)=-n\cdot \mu_{A_i}+\lambda_{A_i}, \quad\text{and}\quad (\varphi)_*(\lambda_K)=\mu_{A_i},$$ and that, since $lk(\gamma^\pm_j,A_i)=1$ and $\psi^\pm$ can be interpreted as a sequence of $n$ crossing changes, then $$\psi^{\pm}_*(\mu_{A_i})=\mu_{U_i}, \quad\text{and}\quad \psi^\pm_*(\lambda_{A_i})=(\mp n)\cdot \mu_{U_i}+\lambda_{U_i}.$$ Similarly, the isomorphism $\theta$ between $S^3\setminus N(U_i)$ and the standard solid torus $D^2\times S^1$ identifies $\mu_{U_i}$ with $l= \left[S^1\right]$ and $\lambda_{U_i}$ with $m=\left[\partial D^2\right]$ so that $$\left(h^\pm_i\right)_*(\mu_K)=m+(-n\mp n)\cdot l\quad\text{ and }\quad
\left(h^\pm_i\right)_*(\lambda_K)=l.$$
Therefore $\left(h^\pm_i\right)_*(\mu_K+(n\pm n)\cdot\lambda_K)=m$ showing that the slope of the surgery is $1/(n\pm n)$. This shows that $$M^+\cong S^3_\frac{1}{2n}(K)\# S^3_\frac{1}{2n}(K)\quad\text{and}\quad M^-\cong S^3_\frac{1}{0}(K)\# S^3_\frac{1}{0}(K)\cong S^3,$$ thus proving that $P_n(K)$ is a cobordism from $\Sigma_n$ to $S^3_\frac{1}{2n}(K)\# S^3_\frac{1}{2n}(K)$, and $R_n(K)$ one from $\Sigma_n$ to $S^3$.\\

To show definiteness, it is enough to understand the intersection form of the 4--manifolds being considered. Let $\{b_1,b_2,\ldots,b_n\}$ be the basis for $H_2(-;\Z)$ determined by the handles. To find a surface that represents $b_j$ consider $S_j$ the Seifert surface for $\gamma_j^+$ in $\Sigma_n$ described in \autoref{seifert_sigma} and push $\text{int}(S_j)$ into the interior of $I\times \Sigma\subset P_n(K)$. Then add the core of the $i$-th handle along $\gamma_j^+$ to obtain a closed surface $\widehat{S_j}$. Next, denote by $Q$ the intersection form of $P_n(K)$. It is well-known that the value of $Q(b_j,b_k)$ is given by the number of points in $\widehat{S_j}\cap \widehat{S_k}$, counted with sign. Then, using \autoref{seifert_sigma} we get
$$\widehat{S_j}\cap \widehat{S_k}=S_j\cap\gamma_k^+=D_j\cap\gamma_k^+.$$
Since the disk $D_j$ is disjoint from every other component of $\boldsymbol{\gamma^+}$, and $\gamma_j^+$ has framing $+1$, the signed number of points in $D_j\cap\gamma_k^+$ is given by the Kronecker delta number $\delta_{ik}$. This shows that the $n\times n$ identity matrix $I_n$ represents the intersection form $Q$ in terms of the basis $\{b_1,b_2,\ldots,b_n\}$, and thus that $P_n(K)$ is a positive definite manifold.\\

The analogous argument applied to $\boldsymbol\gamma^-$ shows that $-I_n$ represents the intersection form of $R_n(K)$ and so that $R_n(K)$ is negative definite.
\end{proof}

\begin{figure}[h]
\medskip
\centering
\begin{subfigure}[b]{0.475\textwidth}
\centering
\begin{tabular}{cc}
\raisebox{0.25em}{\resizebox{!}{3cm}{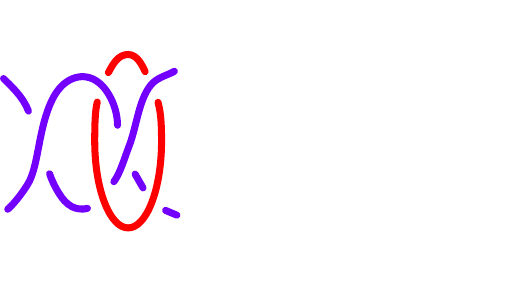}}\hspace*{-3cm}&
\raisebox{2.25em}{\resizebox{!}{1.5cm}{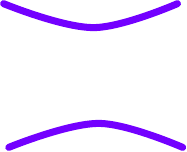}}
\end{tabular}
\caption{One component of $\boldsymbol\gamma^+$ and the corresponding crossing change.}
\end{subfigure}\quad\quad
\begin{subfigure}[b]{0.475\textwidth}
\centering
\begin{tabular}{cc}
\resizebox{5.5cm}{!}{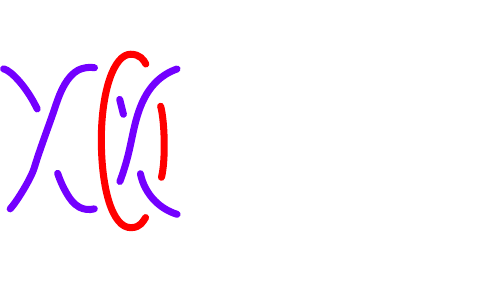}\hspace*{-3cm}&
\raisebox{2.25em}{\resizebox{!}{1.5cm}{\input{strands.pdf_tex}}}
\end{tabular}
\caption{One component of $\boldsymbol\gamma^-$ and the corresponding crossing change.}
\end{subfigure}
\caption{Local depiction of $\gamma^\pm$}\label{crossing}\label{gamma}
\end{figure}

\begin{cor} Let $p,q>0$ and consider the satellite knot $D_n\left(T_{p,q}\right)$. If $\Sigma=\Sigma_2\left(D_n\left(T_{p,q}\right)\right)$ is the 2-fold cover of $S^3$ branched over $D_n\left(T_{p,q}\right)$ then we have the following:
\begin{enumerate}[label=(\alph*),ref=Theorem \thetheorem (\alph*)]
\item\label{pos_torus} There exists a positive definite 4-manifold, $P(n,p,q)$, with boundary components $-\Sigma$ and two copies of $-\Sigma(p,q,2npq-1)$.
\item\label{neg_torus}  There exists a negative definite 4-manifold, $R(n,p,q)$, with boundary $-\Sigma$.
\end{enumerate}
\end{cor}

\begin{proof} First, to construct $P(n,p,q)$ attach a 3--handle to the manifold $P_n(T_{p,q})$ along its outcoming boundary component to transform the connected sum of manifolds into disjoint union. Next, recall that for $p,q,n>0$, the result of $1/2n$ surgery on $S^3$ along the torus knot $T_{p,q}$ is diffeomorphic to the Seifert fibred homology sphere $-\Sigma(p,q,2npq-1)$ \cite[Proposition 3.1]{moser}. \\
Similarly, the manifold $R(n,p,q)$ is obtained from $R_n(T_{p,q})$ by capping off its outcoming boundary component $S^3$ with a 4--ball.
\end{proof}

\section{Main Result}\label{sec::main_result}
\begin{theorem}\label{ind_cob}Let $\left\{\left(p_i,q_i\right)\right\}_i$ be a sequence of relatively prime positive integers and $n_i$ a positive and even integer $(i=1,2,\ldots)$. If $$p_iq_i(2n_ip_iq_i-1)<p_{i+1}q_{i+1}(n_{i+1}p_{i+1}q_{i+1}-1),$$ the family $\mathscr{F}=\left\{\Sigma_2\left(D_{n_i}\left(T_{p_iq_i}\right)\right)\right\}_{i=1}^\infty$ is independent in $\Theta^3_{\Z/2}$.
\end{theorem} 

\begin{proof} Denote by $[Y]$ the homology cobordism class of the $\Z/2$--homology sphere $Y$ and suppose by contradiction that there exist integral coefficients $c_1,\ldots,c_N\in\Z$ such that $$\sum_{i=1}^Nc_i\left[\Sigma_2\left(D_n\left(T_{p_iq_i}\right)\right)\right]=0$$ in $\Theta^3_{\Z/2}$. The supposition implies the existence of an oriented 4-manifold $Q$ with the $\Z/2$ homology of a punctured $4$--ball and with boundary
$$\partial Q=\bigsharpp_{i=1}^N\left(\overset{c_i}{\underset{j=1}{\#}}\Sigma_2\left(D_n\left(T_{p_iq_i}\right)\right)\right).$$
Attaching 3--handles to $Q$ we can further assume that $$\partial Q=\bigsqcup_{i=1}^Nc_i\Sigma_2\left(D_n\left(T_{p_iq_i}\right)\right).$$
Here we use $cY$ to denote the disjoint union of $c$ copies of $Y$ if $c>0$, and $-c$ copies of $-Y$ if $c<0$. In addition, and without loss of generality, further assume that $c_N\geq 1$. Augment $Q$ using the cobordisms constructed in \autoref{cobordisms}, namely, let $$X=Q\cup\bigg(Z(n_N,p_N,q_N)\bigg)\cup\left(\bigsqcup_{c_i>0} R(n_i,p_i,q_i)\right)\cup\left(\bigsqcup_{c_i<0} -P(n_i,p_i,q_i)\right).$$
Recall that $Z(n,p,q)$, $-P(n,p,q)$ and $R(n,p,q)$ are negative definite cobordisms from $\Sigma$ to $-\Sigma(p,q,npq-1)$, $2\left(\Sigma(p,q,2npq-1)\right)$, and the empty set respectively. Thus, $X$ is a negative definite 4--manifold with oriented boundary $$\partial X=\bigg(-\Sigma(p_N,q_N,n_Np_Nq_N-1)\bigg)\sqcup \left(\bigsqcup_{c_i<0}2\Sigma(p_i,q_i,2n_ip_iq_i-1)\right).$$
Additionally, since the first $\Z/2$--homology groups of $Z(n,p,q)$, $-P(n,p,q)$, $R(n,p,q)$, and $Q$ are trivial, the Mayer-Vietoris theorem shows that $H_1(X,\Z/2)=0$. This would imply that the Seifert fibered spaces $\left\{-\Sigma(p_N,q_N,n_Np_Nq_N-1)\right\}\cup \left\{\Sigma(p_i,q_i,2n_ip_iq_i-1)\right\}_{c_i<0}$ cobound a smooth 4--manifold that has negative definite intersection form and that satisfies $H_1(X,\Z/2)=0$, contradicting \autoref{cobound}. Therefore, $Q$ cannot exist and so the 3-manifolds $\Sigma_2\left(D_{n_i}\left(T_{p_iq_i}\right)\right)$ are independent in the $\Z/2$ homology cobordism group.
\end{proof}

\begin{figure}[H]
\medskip
\centering
\includegraphics[width=0.375\textwidth]{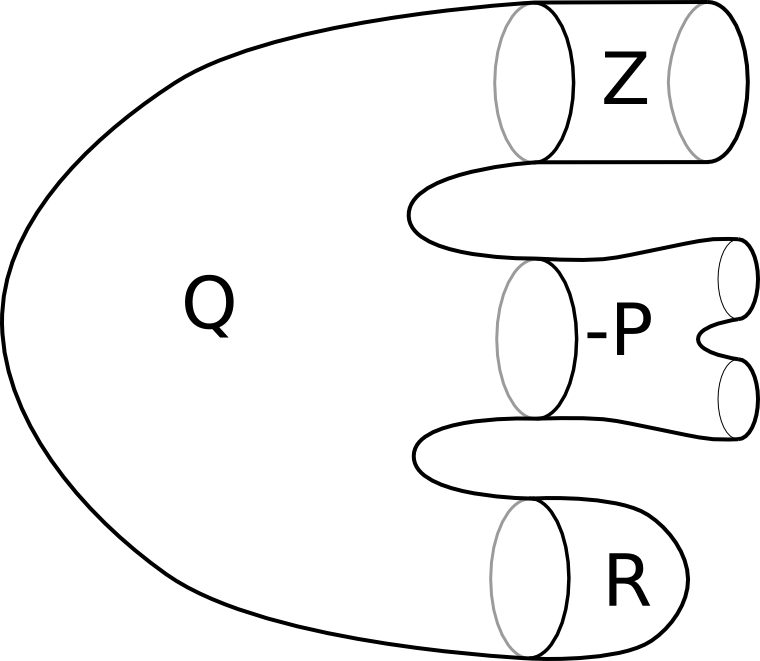}
\caption{The manifold $X$.}
\bigskip
\end{figure}

\begin{theorem}\label{main_result} Let $\left\{\left(p_i,q_i\right)\right\}_i$ be a sequence of relatively prime positive integers and $n_i$ a positive and even integer $(i=1,2,\ldots)$. Then, if $$p_iq_i(2n_ip_iq_i-1)<p_{i+1}q_{i+1}(n_{i+1}p_{i+1}q_{i+1}-1),$$ the collection $\left\{D_{n_i}\left(T_{p_iq_i}\right)\right\}_{i=1}^\infty$ is an independent family in $\Conc_\infty$.
\end{theorem}
\begin{proof} If $c_1D_{n_1}\left(T_{p_1q_1}\right)\#c_2D_{n_2}\left(T_{p_2q_2}\right)\#\ldots\#c_ND_{n_N}\left(T_{p_Nq_N}\right)$ is slice for some integral coefficients $c_1,\ldots,c_N\in\Z$, then \autoref{2-fold_conc} shows that $$\Sigma_2\left(c_1D_{n_1}\left(T_{p_1q_1}\right)\#\ldots\#c_ND_{n_N}\left(T_{p_Nq_N}\right)\right)=c_1\Sigma_2\left(D_{n_1}\left(T_{p_1q_1}\right)\right)\#\ldots\#c_N\Sigma_2\left(D_{n_N}\left(T_{p_Nq_N}\right)\right)$$ is the boundary of a $\Z/2$--homology ball $Q$. However, \autoref{ind_cob} shows that $Q$ does not exist and the result thus follows.\end{proof}

\bibliographystyle{abbrvnat}
\bibliography{References}

\end{document}

%% file: P_n.pdf_tex
\begingroup%
  \makeatletter%
  \providecommand\color[2][]{%
    \errmessage{(Inkscape) Color is used for the text in Inkscape, but the package 'color.sty' is not loaded}%
    \renewcommand\color[2][]{}%
  }%
  \providecommand\transparent[1]{%
    \errmessage{(Inkscape) Transparency is used (non-zero) for the text in Inkscape, but the package 'transparent.sty' is not loaded}%
    \renewcommand\transparent[1]{}%
  }%
  \providecommand\rotatebox[2]{#2}%
  \ifx\svgwidth\undefined%
    \setlength{\unitlength}{504.63282016bp}%
    \ifx\svgscale\undefined%
      \relax%
    \else%
      \setlength{\unitlength}{\unitlength * \real{\svgscale}}%
    \fi%
  \else%
    \setlength{\unitlength}{\svgwidth}%
  \fi%
  \global\let\svgwidth\undefined%
  \global\let\svgscale\undefined%
  \makeatother%
  \begin{picture}(1,0.52516526)%
    \put(0,0){\includegraphics[width=\unitlength]{P_n.pdf}}%
    \put(0.59377778,0.25313172){\color[rgb]{0,0,0}\makebox(0,0)[lb]{\smash{$\ldots$}}}%
    \put(0.12969028,0.25537374){\color[rgb]{0,0,0}\makebox(0,0)[lb]{\smash{$\ldots$}}}%
    \put(0.06177379,0.33030538){\color[rgb]{0,0,0}\makebox(0,0)[lb]{\smash{$+1$}}}%
    \put(0.27238283,0.33030538){\color[rgb]{0,0,0}\makebox(0,0)[lb]{\smash{$+1$}}}%
    \put(0.52651793,0.32964872){\color[rgb]{0,0,0}\makebox(0,0)[lb]{\smash{$+1$}}}%
    \put(0.74029757,0.32978471){\color[rgb]{0,0,0}\makebox(0,0)[lb]{\smash{$+1$}}}%
    \put(0.06390787,0.14222903){\color[rgb]{0,0,0}\makebox(0,0)[lb]{\smash{$\gamma^+_1$}}}%
    \put(0.27792486,0.14222903){\color[rgb]{0,0,0}\makebox(0,0)[lb]{\smash{$\gamma^+_{n/2}$}}}%
    \put(0.52444074,0.14222903){\color[rgb]{0,0,0}\makebox(0,0)[lb]{\smash{$\gamma^+_{n/2+1}$}}}%
    \put(0.74162835,0.14222903){\color[rgb]{0,0,0}\makebox(0,0)[lb]{\smash{$\gamma^+_n$}}}%
    \put(0.38017743,0.47698666){\color[rgb]{0,0,0}\makebox(0,0)[lb]{\smash{$-n$}}}%
    \put(0.38968929,0.0093199){\color[rgb]{0,0,0}\makebox(0,0)[lb]{\smash{$-n$}}}%
  \end{picture}%
\endgroup%

%% file: R_n.pdf_tex
\begingroup%
  \makeatletter%
  \providecommand\color[2][]{%
    \errmessage{(Inkscape) Color is used for the text in Inkscape, but the package 'color.sty' is not loaded}%
    \renewcommand\color[2][]{}%
  }%
  \providecommand\transparent[1]{%
    \errmessage{(Inkscape) Transparency is used (non-zero) for the text in Inkscape, but the package 'transparent.sty' is not loaded}%
    \renewcommand\transparent[1]{}%
  }%
  \providecommand\rotatebox[2]{#2}%
  \ifx\svgwidth\undefined%
    \setlength{\unitlength}{499.14868934bp}%
    \ifx\svgscale\undefined%
      \relax%
    \else%
      \setlength{\unitlength}{\unitlength * \real{\svgscale}}%
    \fi%
  \else%
    \setlength{\unitlength}{\svgwidth}%
  \fi%
  \global\let\svgwidth\undefined%
  \global\let\svgscale\undefined%
  \makeatother%
  \begin{picture}(1,0.53093523)%
    \put(0,0){\includegraphics[width=\unitlength]{R_n.pdf}}%
    \put(0.60459063,0.25591287){\color[rgb]{0,0,0}\makebox(0,0)[lb]{\smash{$\ldots$}}}%
    \put(0.13540423,0.25817952){\color[rgb]{0,0,0}\makebox(0,0)[lb]{\smash{$\ldots$}}}%
    \put(0.06674154,0.33393443){\color[rgb]{0,0,0}\makebox(0,0)[lb]{\smash{$-1$}}}%
    \put(0.27966453,0.33393443){\color[rgb]{0,0,0}\makebox(0,0)[lb]{\smash{$-1$}}}%
    \put(0.5365918,0.33327056){\color[rgb]{0,0,0}\makebox(0,0)[lb]{\smash{$-1$}}}%
    \put(0.75272024,0.33340805){\color[rgb]{0,0,0}\makebox(0,0)[lb]{\smash{$-1$}}}%
    \put(0.06889906,0.1437917){\color[rgb]{0,0,0}\makebox(0,0)[lb]{\smash{$\gamma^-_1$}}}%
    \put(0.28526746,0.1437917){\color[rgb]{0,0,0}\makebox(0,0)[lb]{\smash{$\gamma^-_{n/2}$}}}%
    \put(0.53449179,0.1437917){\color[rgb]{0,0,0}\makebox(0,0)[lb]{\smash{$\gamma^-_{n/2+1}$}}}%
    \put(0.75406564,0.1437917){\color[rgb]{0,0,0}\makebox(0,0)[lb]{\smash{$\gamma^-_n$}}}%
    \put(0.38864347,0.4822273){\color[rgb]{0,0,0}\makebox(0,0)[lb]{\smash{$-n$}}}%
    \put(0.39825984,0.00942229){\color[rgb]{0,0,0}\makebox(0,0)[lb]{\smash{$-n$}}}%
  \end{picture}%
\endgroup%

%% file: gamma_+.pdf_tex
\begingroup%
  \makeatletter%
  \providecommand\color[2][]{%
    \errmessage{(Inkscape) Color is used for the text in Inkscape, but the package 'color.sty' is not loaded}%
    \renewcommand\color[2][]{}%
  }%
  \providecommand\transparent[1]{%
    \errmessage{(Inkscape) Transparency is used (non-zero) for the text in Inkscape, but the package 'transparent.sty' is not loaded}%
    \renewcommand\transparent[1]{}%
  }%
  \providecommand\rotatebox[2]{#2}%
  \ifx\svgwidth\undefined%
    \setlength{\unitlength}{151.0525122bp}%
    \ifx\svgscale\undefined%
      \relax%
    \else%
      \setlength{\unitlength}{\unitlength * \real{\svgscale}}%
    \fi%
  \else%
    \setlength{\unitlength}{\svgwidth}%
  \fi%
  \global\let\svgwidth\undefined%
  \global\let\svgscale\undefined%
  \makeatother%
  \begin{picture}(1,0.54708895)%
    \put(0,0){\includegraphics[width=\unitlength]{gamma_+.pdf}}%
    \put(0.13678372,0.46661197){\color[rgb]{0,0,0}\makebox(0,0)[lb]{\smash{$+1$}}}%
    \put(0.13663096,0.02498096){\color[rgb]{0,0,0}\makebox(0,0)[lb]{\smash{$\gamma^+_1$}}}%
  \end{picture}%
\endgroup%

%% file: strands.pdf_tex
\begingroup%
  \makeatletter%
  \providecommand\color[2][]{%
    \errmessage{(Inkscape) Color is used for the text in Inkscape, but the package 'color.sty' is not loaded}%
    \renewcommand\color[2][]{}%
  }%
  \providecommand\transparent[1]{%
    \errmessage{(Inkscape) Transparency is used (non-zero) for the text in Inkscape, but the package 'transparent.sty' is not loaded}%
    \renewcommand\transparent[1]{}%
  }%
  \providecommand\rotatebox[2]{#2}%
  \ifx\svgwidth\undefined%
    \setlength{\unitlength}{53.71559404bp}%
    \ifx\svgscale\undefined%
      \relax%
    \else%
      \setlength{\unitlength}{\unitlength * \real{\svgscale}}%
    \fi%
  \else%
    \setlength{\unitlength}{\svgwidth}%
  \fi%
  \global\let\svgwidth\undefined%
  \global\let\svgscale\undefined%
  \makeatother%
  \begin{picture}(1,0.80978992)%
    \put(0,0){\includegraphics[width=\unitlength]{strands.pdf}}%
  \end{picture}%
\endgroup%

%% file: gamma_-.pdf_tex
\begingroup%
  \makeatletter%
  \providecommand\color[2][]{%
    \errmessage{(Inkscape) Color is used for the text in Inkscape, but the package 'color.sty' is not loaded}%
    \renewcommand\color[2][]{}%
  }%
  \providecommand\transparent[1]{%
    \errmessage{(Inkscape) Transparency is used (non-zero) for the text in Inkscape, but the package 'transparent.sty' is not loaded}%
    \renewcommand\transparent[1]{}%
  }%
  \providecommand\rotatebox[2]{#2}%
  \ifx\svgwidth\undefined%
    \setlength{\unitlength}{143.91853729bp}%
    \ifx\svgscale\undefined%
      \relax%
    \else%
      \setlength{\unitlength}{\unitlength * \real{\svgscale}}%
    \fi%
  \else%
    \setlength{\unitlength}{\svgwidth}%
  \fi%
  \global\let\svgwidth\undefined%
  \global\let\svgscale\undefined%
  \makeatother%
  \begin{picture}(1,0.5785167)%
    \put(0,0){\includegraphics[width=\unitlength]{gamma_-.pdf}}%
    \put(0.14702523,0.49405051){\color[rgb]{0,0,0}\makebox(0,0)[lb]{\smash{$-1$}}}%
    \put(0.14681552,0.02621926){\color[rgb]{0,0,0}\makebox(0,0)[lb]{\smash{$\gamma^-_1$}}}%
  \end{picture}%
\endgroup%